\documentclass[a4paper,english,11pt]{amsart}
\usepackage{babel}
\usepackage{color}
\usepackage{graphics}
\usepackage{graphicx}
\usepackage{amsmath}
\usepackage{fancybox}
\usepackage[latin1]{inputenc}
\usepackage{textcomp}
\usepackage{stmaryrd}
\usepackage{cases}
\usepackage{pgf}
\usepackage[all]{xy}
\usepackage{pgf,amssymb,amsmath}
\usepackage{amssymb,amsmath,amstext,stmaryrd,euscript,tabularx,epsfig}
\usepackage{geometry}
\usepackage{fancyhdr}
\usepackage{pstricks}

\newcommand{\Aut}{\mathop{\mathrm{Aut}}}
\newcommand{\supp}{\textup{supp}}
\newcommand{\id}{\mathop{\mathrm{id}}}

\newcommand{\R}{\mathbb{R}}
\newcommand{\D}{\mathbb{D}}
\newcommand{\B}{\mathbb{B}}
\newcommand{\F}{\mathcal{F}}
\newcommand{\J}{\mathcal{J}}

\newcommand{\C}{\mathbb{C}}
\newcommand{\p}{\mathbb{P}}
\newcommand{\pe}{\textup{ := }}
\newcommand{\rat}{\textup{Rat}}
\newcommand{\Per}{\textup{Per}}
\newcommand{\bif}{\textup{bif}}
\newcommand{\poly}{\textup{Poly}}

\begin{document}

\def\theequation{\thesection.\arabic{equation}}
\def\sqw{\hbox{\rlap{\leavevmode\raise.3ex\hbox{$\sqcap$}}$%
\sqcup$}}
\def\sqb{\hbox{\hskip5pt\vrule width4pt height4pt depth1.5pt%
\hskip1pt}}

\newtheorem{tm}{Theorem}[section]
\newtheorem{defi}[tm]{Definition}
\newtheorem{prop}[tm]{Proposition}
\newtheorem{propdefi}[tm]{Proposition-Definition}
\newtheorem{nota}[tm]{Notation}
\newtheorem{rem}[tm]{Remark}
\newtheorem{lm}[tm]{Lemma}
\newtheorem{cor}[tm]{Corollary}
\newtheorem{ctrex}[tm]{Countre-example}
\newtheorem{apl}[tm]{Application}
\newtheorem{fait}{Fact}

\title{Strong bifurcation loci of full Hausdorff dimension}
\begin{author}[Thomas~Gauthier]{Thomas Gauthier}
\email{thomas.gauthier$@$math.univ-toulouse.fr}
\address{Universit\'e Paul Sabatier\\
  Institut de Math\'ematiques de Toulouse \\
  118, route de Narbonne \\
  31062 Toulouse Cedex \\
  France }
\end{author}

\maketitle

\begin{abstract}
In the moduli space $\mathcal{M}_d$ of degree $d$ rational maps, the bifurcation locus is the support of a closed $(1,1)$ positive current $T_\bif$ which is called the bifurcation current. This current gives rise to a measure $\mu_\bif:=(T_\bif)^{2d-2}$ whose support is the seat of strong bifurcations. Our main result says that $\supp(\mu_\bif)$ has maximal Hausdorff dimension $2(2d-2)$. As a consequence, the set of degree $d$ rational maps having $2d-2$ distinct neutral cycles is dense in a set of full Hausdorff dimension.
\end{abstract}

\section{Introduction}
\footnote{Math Subject Class : 37F45; 32U15; 28A78} 

\par The boundary of the Mandelbrot set has Hausdorff dimension 2. This fundamental result is the main Theorem of Shishikura's work \cite{Shishikura2}. Tan Lei has generalized this by showing that the boundary of the connectedness locus of polynomial families of any degree has maximal Hausdorff dimension. Tan Lei has also shown that the bifurcation locus in any non-stable holomorphic family of rational maps is of full dimension (see \cite{TanLei}). McMullen gave another proof of her result in \cite{McMullen3}. Our aim here is to show that dynamically relevent, but a priori much smaller, subsets of the bifurcation locus have maximal Hausdorff dimension in the space $\rat_d$ of all degree $d$ rational maps.
\par We can define a \emph{bifurcation current} on $\rat_d$ by setting $T_\bif\pe dd^cL$, where $L(f)$ is the Lyapounov exponent of $f$ with respect to its maximal entropy measure. DeMarco has shown in \cite{DeMarco1} and \cite{DeMarco2} that the support of $T_\bif$ is precisely the bifurcation locus. This current and its powers $T_\bif^k$ ($k\leq 2d-2$) have been used in several recent works for studying the geometry of the bifurcation locus (see \cite{BB1,BB3,BB2,buffepstein,dinhsibony2,dujardin2,favredujardin}). Moebius transformations act by conjugacy on $\rat_d$ and the quotient space is an orbifold known as the \emph{moduli space} $\mathcal{M}_d$ of degree $d$ rational maps. Giovanni Bassanelli and Fran\c{c}ois Berteloot \cite{BB1} introduced a measure $\mu_\bif$ on this moduli space, which may be obtained by pushing forward $T_\bif^{2d-2}$. We will call \emph{strong bifurcation locus} the support of this measure $\mu_\bif$. This set can be interpreted as a set on which bifurcations are maximal. Our main result is the following:

~

\begin{tm}
The support of the bifurcation measure $\mu_\bif$ of the moduli space $\mathcal{M}_d$ of degree $d$ rational maps is homogeneous and has maximal Hausdorff dimension, i.e.
\begin{center}
$\dim_H(\supp(\mu_\bif)\cap\Omega)=2(2d-2)$
\end{center}
for any open set $\Omega\subset\mathcal{M}_d$ such that $\supp(\mu_\bif)\cap\Omega\neq\emptyset$.
\label{tmprincipal}
\end{tm}

\par Let us mention that this implies that the conjugacy classes of rational maps having $2d-2$ distinct neutral cycles are dense in a homogeneous set of full Hausdorff dimension in $\mathcal{M}_d$ (see Main Theorem of \cite{buffepstein}).

~

\par As we shall now explain, this will be obtained by using Misiurewicz rational maps properties and bifurcation currents techniques. A rational map $f$ is $k$-\emph{Misiurewicz} if its Julia set contains exactly $k$ critical points counted with multiplicity, if $f$ has no parabolic cycle and if the $\omega$-limit set of any critical point in its Julia set does not meet the critical set. A classical result of Ma\~n\'e states that the $k$ critical points of $f$ which are in $\J_f$ eventually fall under iteration in a compact $f$-hyperbolic set $E_0$, which means that $f$ is uniformly expanding on $E_0$, and that the $2d-2-k$ remaining critical points are in attracting basins of $f$ (see Section \ref{sectionMis}).

\par The first result we need to establish is the following transversality theorem:

~

\begin{tm}[Weak transversality]
Let $(f_\lambda)_{\lambda\in\B(0,r)}$ be a holomorphic family of degree $d$ rational maps parametrized by a ball $\B(0,r)\subset\C^{2d-2}$ with $2d-2$ marked critical points. Let $f_0$ be $k$-Misiurewicz but not a flexible Latt\`es map. Let us denote by $E_0$ the compact $f_0$-hyperbolic set such that $f_0^{k_0}(c_1(0)),\ldots,f_0^{k_0}(c_k(0))\in E_0$. Denote also by $h$ the dynamical holomorphic motion of $E_0$. If the set $\{\lambda \ / \ \exists m\in\Aut(\p^1),f_\lambda\circ m=m\circ f_{\lambda_0}\}$ is discrete for any $\lambda_0\in\B(0,r)$, then 
\begin{center}
$\textup{codim}\ \{\lambda\in\B(0,r) \ / \ f_\lambda^{k_0}(c_j(\lambda))=h_\lambda(f_0^{k_0}(c_j(0))),1\leq j\leq k\}=k$.
\end{center}
\label{tmtransverseintro}
\end{tm}

\par The establishment of Theorem \ref{tmtransverseintro} is the subject of Section 3. Let us mention that a stronger transversality result has been proved by van Strien in \cite{vanstrien} in the case of $(2d-2)$-Misiurewicz map with a trivial stabilizer for the action by conjugaison of the group $\Aut(\p^1)$ on $\rat_d$, and by Buff and Epstein in \cite{buffepstein} in the case of strictly postcritically finite rational maps. We want here to give a weaker result which is easier to prove and sufficient for our purpose. In their work \cite{buffepstein}, Buff and Epstein established their tranversality Theorem using quadratic differential techniques. As these tools are not well adapted when the critical orbits are infinite, we have instead followed van Strien's and Aspenberg's ideas (see \cite{vanstrien} and \cite{Aspenberg1}) of using quasiconformal maps. See also \cite{Rivera} for a transversality result concerning quadratic semihyperbolic polynomials.

~

\par Section 4 is devoted to local dimension estimates. To achieve our goal we have to prove that the set of $(2d-2)$-Misiurewicz maps has maximal Hausdorff dimension in $\rat_d$. Like in Shishikura's or Tan Lei's work, this basically requires to "copy" big hyperbolic sets in the parameter space. For this, using the transversality Theorem \ref{tmtransverseintro}, we construct a transfer map from the dynamical plane to the parameter space which enjoys good regularity properties. When $k=1$, our proof here is actually slightly simpler than the classical ones. Indeed, in his original proof, Shishikura builds two successive holomorphic motions, where we only use one such motion. Using the work of Shishikura \cite{Shishikura2} on parabolic implosion and Theorem \ref{tm2intro}, and using again the transversality Theorem \ref{tmtransverseintro}, we deduce that the set $\mathfrak{M}_k$ is homogeneous and has maximal Hausdorff dimension $2(2d+1)$. The main result of Section 4 is the following:

~

\begin{tm}
Let $1\leq k\leq 2d-2$.  Denote by $\mathfrak{M}_k$ the set of all $k$-Misiurewicz degree $d$ rational maps with simple critical points which are not flexible Latt\`es maps. Then for any $f\in\mathfrak{M}_k$ and any neighborhood $V_0\subset\rat_d$ of $f$ one has:
\begin{eqnarray*}
\dim_H(\mathfrak{M}_k\cap V_0)\geq 2(2d+1-k)+k\dim_{\textup{hyp}}(f).
\end{eqnarray*}
\label{tm2intro}
\end{tm}

\par For our purpose we have to show that $k$-Misiurewicz maps belong to the support of $T_\bif^k$. To achieve this goal, we give a criterion for a rational map with $k$ critical points eventually falling under iteration in a compact hyperbolic set to belong to the support of $T_\bif^k$. The condition is exactly the one that appears in the statement of the Transversality Theorem \ref{tmtransverseintro}. Our approach is inspired by the work of Buff and Epstein \cite{buffepstein} who proved that strictly postcritically finite maps, which are $(2d-2)$-Misiurewicz, belong to the support of $T_\bif^{2d-2}$. The main idea of the proof of this criterion is a dynamical renormalization process which enlights a similarity phenomenon between parameter and dynamical spaces at maps with critical points eventually falling into a hyperbolic set. To implement this process we need to linearize the map along a repelling orbit. In the geometrically finite case this is just the linearization along repelling cycles. This technical result, which may be of independent interest, is established in Section 5. The renormalization is then performed in Section 6. In the same Section, we also prove that the $k^\text{th}$-self-intersection of the bifurcation current detects the activity of at least $k$ distinct critical points (see Theorem \ref{tmddcg}). Combining the results of Section 6 with Theorem \ref{tmtransverseintro}, we get the following:

~

\begin{tm}
Let $1\leq k\leq 2d-2$ and $f\in\rat_d$ be $k$-Misiurewicz but not a flexible Latt\`es map. Then $f\in\supp(T_\bif^k)\setminus\supp(T_\bif^{k+1})$.
\label{tm1intro}
\end{tm}

\par In Section 7, we focus on dimension estimates for the bifurcation measure. The estimates established in Section 4 have the following interesting consequence. Due to Theorem \ref{tm1intro}, they directly imply that the set $\supp(T_\bif^k)\setminus\supp(T_\bif^{k+1})$ has Hausdorff dimension $2(2d+1)$ for $1\leq k\leq 2d-3$, and that $\supp(T_\bif^{2d-2})$ is homogeneous and that the set $\dim_H(\supp(T_\bif^{2d-2}))=2(2d+1)$ (see Theorem \ref{tmdimH3}). Notice that it is still unknown wether $\supp(T_\bif^k)$ is homogeneous or not, when $1\leq k\leq 2d-3$. Theorem \ref{tmprincipal} then immediatly follows. Recall that, if $\mu$ is a Radon measure on a metric space, the \emph{upper pointwise dimension} of $\mu$ at $x_0$ is
\begin{center}
$\overline{\dim}_{\mu}(x_0)\pe\displaystyle\limsup_{r\rightarrow0}\frac{\log\mu(\B(x_0,r))}{\log r}$.
\end{center}
In the same Section, we also establish a lower bound for the upper pointwise dimension of the bifurcation measure at $(2d-2)$-Misiurewicz parameters, which rely on the renormalization process performed in Section 6. In particular, we get the following:

~

\begin{tm}
There exists a dense subset $\mathfrak{M}$ of $\supp(\mu_\bif)$ which is homogeneous and has maximal Hausdorff dimension $2(2d-2)$ such that $\overline{\dim}_{\mu_\bif}[f]>0$ for any $[f]\in\mathfrak{M}$.
\label{tm3intro}
\end{tm}

\par Let us mention that all these results have their counterpart in poynomial families. We develop the description of the above results in the context of polynomial families in Section 8. For instance, the equivalent of Theorem \ref{tmprincipal} might be stated as follows:

~

\begin{tm}
In the moduli space $\mathcal{P}_d$ of degree $d$ polynomials, the Shilov boundary of the connectedness locus is homogeneous and has maximal Hausdorff dimension $2(d-1)$.
\label{tmshilov}
\end{tm}

\par This improves Tan Lei's result. We also give a very easy proof of Theorem \ref{tmtransverseintro} in the case of a degree $d$ polynomial with $k$ critical points which are preperiodic to repelling cycles (see Lemma \ref{lmpolynomes}).

\subsection{Notation} 
To end the introduction, let us give some notation:
\begin{enumerate}
\item[-] $\p^1$ is the Riemann sphere, $\D$ is the unit disc of $\C$ and $\D(0,r)$ is the disc of $\C$ of radius $r$ centered at $0$. $\textup{Leb}$ denotes the Lebesgue measure on $\p^1$.
\item[-] $\|\cdot\|$ is a norm on $\C^n$ and $\B(a,r)$ is the ball of $\C^n$ of radius $r$ centered at $a\in\C^n$.
\item[-] $\J_f$ is the  Julia set and $\F_f$ is the Fatou set of $f\in\rat_d$.
\item[-] The group $\Aut(\p^1)$ is the group of all Moebius transformations. The quotient space $\mathcal{M}_d\pe\rat_d/\textup{Aut}(\p^1)$ is the \emph{moduli space} of degree $d$ rational maps and $\Pi:\rat_d\longrightarrow\mathcal{M}_d$ is the quotient map. The space $\mathcal{M}_d$ has a canonical structure of affine variety of dimension $2d-2$ (see \cite{Silverman} pages $174-179$).
\item[-] Finally, $\dim_H$ is the Hausdorff dimension in any metric space.
\end{enumerate}

\subsection{Acknowledgment} The author would very much like to thank Fran\c{c}ois Berteloot for his help and encouragement for writting this article. The author also wants to thank Romain Dujardin for his helpful comments and for suggesting Lemma \ref{lmpolynomes} and the formulation of Theorem \ref{tmcourantsbif}. Finally, we would like to thank the Referees for their very  useful suggestions, which have greatly improved the readability of this work.

\section{The framework}

\subsection{Hyperbolic sets}

~

\begin{defi}
Let $f\in\rat_d$ and $E\subset\p^1$ be a compact $f$-invariant set, i.e. such that $f(E)\subset E$. We say that $E$ is $f$-\emph{hyperbolic} if one of the following equivalent conditions is satisfied:
\begin{enumerate}
\item there exists constants $C>0$ and $\alpha>1$ such that $|(f^n)'(z)|\geq C\alpha^n$ for all $z\in E$ and all $n\geq0$,
\item for some appropriate metric on $\p^1$, there exists $K>1$ such that $|f'(z)|\geq K$ for all $z\in E$. One says that $K$ is the \emph{hyperbolicity constant} of $E$.
\end{enumerate}
\end{defi}

~

\par The following Theorem is now classical (see \cite{demelovanstrien} Theorem 2.3 page 225 or \cite{Shishikura2} section 2 for a sketch of proof):

~

\begin{tm}[de Melo-van Strien]
Let $(f_\lambda)_{\lambda\in\mathbb{B}(0,r)}$ be a holomorphic family of degree $d$ rational maps parametrized by a ball $\mathbb{B}(0,r)\subset\C^m$. Let $E_0\subset\p^1$ be a compact $f_0$-hyperbolic set. Then there exists $0<\rho\leq r$ and a unique holomorphic motion 
\begin{eqnarray*}
h:\mathbb{B}(0,\rho)\times E_0 & \longrightarrow & \p^1\\
(\lambda,z) & \longmapsto & h_\lambda(z)
\end{eqnarray*}
which conjugates $f_0$ to $f_\lambda$ on $E_0$, i.e. such that for all $\lambda\in\mathbb{B}(0,\rho)$ the diagram:
\begin{center}
$\xymatrix {\relax
E_0 \ar[r]^{f_0} \ar[d]_{h_\lambda} & E_0\ar[d]^{h_\lambda} \\
E_\lambda \ar[r]_{f_\lambda} & E_\lambda }$
\end{center}
commutes ($E_\lambda\pe h_\lambda(E_0)$). We shall call $h$ the \emph{dynamical holomorphic motion} of $E_0$.
\label{mvtholo}\nocite{vanstrien}
\end{tm}

~

\begin{rem}
If $\B(\lambda_0,\varepsilon)\subset\B(0,\rho)$, one may check that there exists a unique holomorphic motion $g:\B(\lambda_0,\varepsilon)\times E_{\lambda_0}\longrightarrow\p^1$ such that
\begin{center}
$h_\lambda(z)=g_\lambda\circ h_{\lambda_0}(z)$
\end{center}
for all $z\in E_0$ and all $\lambda\in \B(\lambda_0,\varepsilon)$.
\label{Remvtholo}
\end{rem}
~

\par Let $K>1$ be the hyperbolicity constant of $E_0$ and $\mathcal{N}_\delta$ be a $\delta$ neighborhood of $E_0$. Up to reducing $r$, $\delta$ and $K$ we may assume that
\begin{center}
$|f_\lambda'(z)|\geq K$ for every $\lambda\in\B(0,r)$ and $z\in\mathcal{N}_\delta$.
\end{center}
Set $B\pe \sup_{\lambda\in\B(0,r), \ z\in\mathcal{N}_\delta}|f_\lambda'(z)|$. Let us stress that Theorem \ref{mvtholo} implicitly contains the existence of inverse branches of $f_\lambda$ along $E_\lambda$:

~

\begin{lm}
Under the hypothesis of Theorem \ref{mvtholo}, there exists $\varepsilon>0$ such that for all $w_0\in E_0$ and all $k\geq1$, there exists an inverse branch $f_{f_0^{k-1}(w_0),\lambda}^{-1}(z)$ of $f_\lambda$ defined on $\B(0,r)\times\D(f^{k}_0(w_0),\varepsilon)$, taking values in $\D(f^{k-1}_0(w_0),\varepsilon)$ and such that:
\begin{enumerate}
	\item $h_\lambda(f_0^{k-1}(w_0))=f_{f_0^{k-1}(w_0),\lambda}^{-1}\circ h_\lambda(f_0^k(w_0))$ for $\lambda\in\B(0,r)$,
	\item $\frac{1}{B}|z-w|\leq|f_{f_0^{k-1}(w_0),\lambda}^{-1}(z)-f_{f_0^{k-1}(w_0),\lambda}^{-1}(w)|\leq\frac{1}{K}|z-w|$ for $\lambda\in\B(0,r)$ and $z,w\in\D(f_0^{k}(w_0),\varepsilon)$.
\end{enumerate}
\label{cormvtholo}
\end{lm}

\subsection{Misiurewicz rational maps}\label{sectionMis}

\par Let us recall that for $z\in\p^1$ and $f\in\rat_d$ the $\omega$-\emph{limit set} of $z$ is given by $\omega(z)\pe\bigcap_{n\geq0}\overline{\{f^k(z) \ / \ k>n\}}$ and that $z$ is \emph{recurrent} if $z\in\omega(z)$.

~

\begin{defi}
Let $f\in\rat_d$. We say that $f$ is \emph{Misiurewicz} if $f$ has no parabolic cycle, $C(f)\cap\mathcal{J}_f\neq\emptyset$ and $\omega(c)\cap C(f)=\emptyset$ for every $c\in C(f)\cap\mathcal{J}_f$. Moreover, when $\J_f$ contains exactly $k$ critical points of $f$ counted with multiplicity we say that $f$ is $k$-\emph{Misiurewicz}.
\end{defi}

~

\par The three following Theorems of Ma\~n\'e (see \cite{Shishikura} Theorems 1.1, 1.2 and 1.3 page 266) are among the main tools for studying the dynamics of Misiurewicz rational maps.

~

\begin{tm}[Local version]
Let $f\in\rat_d$. There exists an integer $N\geq1$ not depending on $f$ for which if $x\in\J_f$ is not a parabolic periodic point of $f$ and $x$ is not contained in the $\omega$-limit set of any recurrent critical point of $f$, then for every $\varepsilon>0$ there exists a neighborhood $U$ of $x$ in $\p^1$ such that for every $n\geq0$ and every component $V$ of $f^{-n}(U)$:
\begin{enumerate}
\item $\textup{diam}(V)\leq\varepsilon$ and $\deg(f^n:V\longrightarrow U)\leq N$,
\item for every $\varepsilon_1$ there exists $n_0\geq1$ for which $\textup{diam}(V)\leq\varepsilon_1$ as soon as $n\geq n_0$.
\end{enumerate} 
\label{mane1}
\end{tm}

~

\begin{tm}[Compact version]
Let $f\in\rat_d$ and $K\subset\J_f$ be a $f$-invariant compact set not containing critical points nor parabolic periodic points of $f$. If $K$ doesn't meet the $\omega$-limit set of any recurrent critical point of $f$, then $K$ is $f$-hyperbolic.
\label{mane2}
\end{tm}

~

\begin{tm}[An application]
Let $f\in\rat_d$ and $\Gamma$ be either a Cremer cycle, or the boundary of a Siegel disc or a connected component of the boundary of a Herman ring of $f$. Then there exists a recurrent critical point $c$ of $f$ for which $\Gamma\subset\omega(c)$.
\label{mane3}
\end{tm}

~

\par Recall that $f\in\rat_d$ is said to be \emph{semi-hyperbolic} if there exists $\delta>0$ and $d_0\in\mathbb{N}^*$ such that for any $z\in\J_f$ and $n\in\mathbb{N}$ we have $\deg(f^n:U(z,f^n,\delta)\longrightarrow\D(f^n(z),\delta))\leq d_0$, where $U(z,f^n,\delta)$ is the component of $f^{-n}(\D(f^n(z),\delta))$ containing $z$ (see \cite{rudiments} pages 96-97).

~

\par Let $f\in\rat_d$. For $k\geq0$ we set
\begin{center}
$P^k(f)\pe\overline{\{f^n(c) \ / \ n\geq k\text{ and }c\in\J_f\cap C(f)\}}$.
\end{center}
\par The items $1$, $2$, $3$ and $4$ of the next Proposition follow from Ma\~n\'e's Theorems:

~

\begin{prop}
Let $f\in\rat_d$ be Misiurewicz. Then 
\begin{enumerate}
\item there exists $k_0\geq1$ such that $P^{k_0}(f)$ is a compact $f$-hyperbolic set,
\item $f$ has no neutral cycle,
\item the periodic Fatou components of $f$ are attracting basins,
\item $f$ is semi-hyperbolic and in particular, either $\J_f=\p^1$ or $\textup{Leb}(\J_f)=0$,
\item if $f$ carries an invariant line field on its Julia set, then $f$ is a flexible Latt\`es map.
\end{enumerate}
\label{fatou}
\end{prop}

~

Item $(5)$ uses the fact that the set of conic points of $f$ coincides with $\J_f$ when $f$ is semi-hyperbolic (see \cite{rudiments} page 109 and \cite{Urbanski} Proposition 6.1). In particular if $f$ is Misiurewicz and carries an invariant line field on its Julia set, item $(4)$ implies that $f$ is a Latt\`es map (see \cite{rudiments} Theorem VII.22). More precisely, $f$ is a flexible Latt\`es map (see \cite{McMullen} corollary 3.18).

\subsection{Bifurcation currents}

\par Let $(f_\lambda)_{\lambda\in X}$ be a holomorphic family of degree $d$ rational maps. The theory of bifurcations of a holomorphic family, introduced by Ma\~n\'e, Sad and Sullivan, makes a link between the instability of critical orbits, and that of neutral or repelling cycles. To make this precise, we set:
\begin{center}

~

$\mathcal{R}(X)\pe\{\lambda_0\in X$ / every repelling cycle of $f_{\lambda_0}$ moves holomorphically on a fixed neighborhood of $\lambda_0\}$
 
~

$\mathcal{S}(X)=\{\lambda_0\in X$ / $(\lambda\longmapsto f_\lambda^n(C(f_\lambda)))_{n\geq0}$ is equicontinuous at $\lambda_0\}$

~

$\mathcal{N}(X)\pe\{\lambda_0\in X$ / $f_{\lambda_0}$ has a non-persistent neutral cycle$\}$.

\end{center}

~

\par The following Theorem is due to Ma\~n\'e, Sad and Sullivan (\cite{MSS}):

~

\begin{tm}[Ma\~n\'e-Sad-Sullivan]
Let $(f_\lambda)_{\lambda\in X}$ be a holomorphic family of degree $d$ rational maps. Then $\mathcal{R}(X)=\mathcal{S}(X)=X\setminus\overline{\mathcal{N}(X)}$ is an open dense subset of $X$.
\label{tmMSS}
\end{tm}

~

\par The set $X\setminus\mathcal{R}(X)$ is called the \emph{bifurcation locus} of the family $(f_\lambda)_{\lambda\in X}$.

~

\begin{defi}
A critical point $c$ is said to be \emph{marked} if there exists a holomorphic function $c:X\longrightarrow\p^1$ satisfying $f_\lambda'(c(\lambda))=0$ for every $\lambda\in X$.
\par We say that the critical point $c$ is \emph{active} at $\lambda_0\in X$ if $(f_\lambda^n(c(\lambda)))_{n\geq0}$ is not a normal family in any neighborhood of $\lambda_0$. Otherwise we say that $c$ is \emph{passive} at $\lambda_0$. The \emph{activity locus} of $c$ is the set of parameters $\lambda\in X$ at which $c$ is active.\label{actif}
\end{defi}

~

\par When $(f_\lambda)_{\lambda\in X}$ is a holomorphic family of degree $d$ rational maps with $2d-2$ marked critical points $c_1,\ldots,c_{2d-2}$, Theorem \ref{tmMSS} states that the bifurcation locus coincides with the union of the activity loci of the $c_i$'s.

~

\par Recall that $f\in\rat_d$ admits a unique maximal entropy measure $\mu_f$. The \emph{Lyapounov exponent} of $f$ with respect to the measure $\mu_f$ is the real number $L(f)\pe\int_{\p^1}\log|f'|\mu_f$. For a holomorphic family $(f_\lambda)_{\lambda\in X}$ of degree $d$ rational maps, we denote by $L(\lambda)\textup{:=} L(f_\lambda)$. Then, the function $\lambda \longmapsto L(\lambda)$ is called the \emph{Lyapounov function} of the family $(f_\lambda)_{\lambda\in X}$. It is a plurisubharmonic (which we will denote by $p.s.h$) and continuous function on $X$ (see \cite{BB1} Corollary 3.4). The Margulis-Ruelle inequality implies that $L(f)\geq\frac{\log d}{2}$.

~

\begin{defi}
The \emph{bifurcation current} of the family $(f_\lambda)_{\lambda\in X}$ is the $(1,1)$-closed positive current on $X$ defined by $T_{\textup{bif}}\pe dd^cL(\lambda)$.
\end{defi}

~

The support of $T_{\textup{bif}}$ coincides with the bifurcation locus of the family $(f_\lambda)_{\lambda\in X}$ in the sense of Ma\~n\'e-Sad-Sullivan. This actually follows from the so-called DeMarco's formula which we now present (see \cite{DeMarco2} Theorem 1.1 or \cite{BB1} Theorem 5.2).

~

\par Let $\pi:\C^2\setminus\{0\}\longrightarrow\p^1$ be the canonical projection. Every $f\in\rat_d$ admits a lift $F:\C^2\longrightarrow\C^2$ under $\pi$. The map $F$ is a degree $d$ non-degenerate homogeneous polynomial endomorphism of $\C^2$. The \emph{Green function} of a lift $F$ of $f$ is defined by
\begin{center}
$G_F\pe\lim_{n\rightarrow\infty}d^{-n}\log\|F^n\|$. 
\end{center}
The following properties of the Green function will be useful (see \cite{BB1} Proposition 1.2).

~

\begin{prop}
Let $(f_\lambda)_{\lambda\in X}$ be a holomorphic family of degree $d$ rational maps which admits a holomorphic family of lifts $(F_\lambda)_{\lambda\in X}$. The function $G_\lambda(z)$ is $p.s.h$ and continuous on $X\times(\C^2\setminus\{0\})$. Moreover it satisfies the homogeneity property
\begin{eqnarray}
G_\lambda(tz)=\log|t|+G_\lambda(z)
\label{homogeneityGreen}
\end{eqnarray}
for every $\lambda\in X$ and $t\in\C\setminus\{0\}$ and the functional equation
\begin{eqnarray}
G_\lambda\circ F_\lambda=dG_\lambda.
\label{functorialGreen}
\end{eqnarray}
\label{propGreen}
\end{prop}

If $U$ is a small ball in $X$, there exists a lift $\tilde c_i(\lambda)$ of the marked critical points $c_i(\lambda)$ under the projection $\pi$ and a holomorphic family of lifts $(F_\lambda)_{\lambda\in U}$. DeMarco's formula may then be stated as:
\begin{eqnarray}
T_{\textup{bif}}\big|_U=\displaystyle\sum_{j=1}^{2d-2}dd^cG_\lambda(\tilde c_j(\lambda)).
\label{Demarco2}
\end{eqnarray}
Let us stress that the support of $dd^cG_\lambda(\tilde c_i(\lambda))$ is precisely the activity locus of $c_i$ in $U$. For more details on pluripotential theory, see \cite{dinhsibony2}.

~

\section{Activity and weak transversality}

~

\par To any $k$-Misiurewicz rational map $f_0$ within a holomorphic family $(f_\lambda)_{\lambda\in\B(0,r)}$ is naturally associated an analytic subset $X_0$ of $\B(0,r)$ such that $f_\lambda$ is $k$-Misiurewicz for all $\lambda\in X_0$ and $(f_\lambda)_{\lambda\in X_0}$ is stable. In order to describe this set, we introduce a holomorphic map $\chi:\B(0,r)\longrightarrow\C^ k$ which we call the activity map of $(f_\lambda)_{\lambda\in\B(0,r)}$ at $\lambda=0$. We show in particular that $X_0=\chi^{-1}\{0\}$ has codimension $k$. The map $\chi$ will turn out to induce local "changes of coordinates" which will play a crucial role in our study of the bifurcation locus near Misiurewicz parameters.

\subsection{The activity map}\label{sectionactivitymap}

\par Let $(f_\lambda)_{\lambda\in\B(0,r)}$ be a holomorphic family of degree $d$ rational maps with $2d-2$ marked critical points $c_1,\ldots,c_{2d-2}$ and parametrized by a ball $\B(0,r)\subset\C^{2d-2}$. Assume that $c_1(0),\ldots,c_k(0)\in\J_0$ and that
\begin{center}
$E_0\pe\overline{\{f^n_0(c_j(0))/ n\geq k_0,1\leq j\leq k\}}$
\end{center}
is a $f_0$-hyperbolic set (which is the case if $f_0$ is $k$-Misiurewicz). Denote by $h:\B(0,r)\times E_0\longrightarrow\p^1$ the dynamical holomorphic motion of $E_0$ (see item 1. of Proposition \ref{fatou}). For $\lambda\in\B(0,r)$, $n\geq0$ and $1\leq i\leq k$ we will adopt the following notations:

\begin{center}
$\begin{array}{clc}
& \xi_{n,i}(\lambda)\pe f_\lambda^{n+k_0}(c_i(\lambda)),& \\
& & \\
& \nu_{n,i}(\lambda)\pe h_\lambda\big(f_0^{n+k_0}(c_i(0))\big)=f_\lambda^n\circ h_\lambda\big(f_0^{k_0}(c_i(0))\big),&\\
& & \\
& m_{n,i}(\lambda)\pe (f_\lambda^n)'\big(\nu_{0,i}(\lambda)\big), &\\
& & \\
& \chi_i(\lambda)\pe\xi_{0,i}(\lambda)-\nu_{0,i}(\lambda). &
\end{array}$
\end{center}

~

\par To define correctly $\chi_i$ one should use a local chart around $\xi_{0,i}(0)=\nu_{0,i}(0)$. Moreover, we will identify the tangent spaces $T_{\nu_{0,i}(\lambda)}\p^1$ and $T_{\nu_{n,i}(\lambda)}\p^1$ with $\C$ to view $m_{n,i}(\lambda)$ as a complex number. As the next Lemma shows, the function $\chi_i$ describes the activity locus of the critical point $c_i$:

~

\begin{lm}[Activity]
Let $(f_\lambda)_{\lambda\in\mathbb{B}(0,r)}$ be as above. Then, for each $1\leq i\leq k$,
\begin{enumerate}  
  \item as a marked critical point of the restricted family $(f_\lambda)_{\lambda\in \chi_i^{-1}\{0\}}$, $c_i$ is passive at every $\lambda\in \chi_i^{-1}\{0\}$,
  \item $\xi_{0,i}\not\equiv\nu_{0,i}$ on $\mathbb{B}(0,r)$ if and only if $c_i$ is active at every $\lambda_0\in \chi_i^{-1}\{0\}$ in the full family $(f_\lambda)_{\lambda\in X}$.
\end{enumerate}\label{lmactivite}
\end{lm}

\begin{proof} If $\chi_i(\lambda)=0$ we have $\xi_{0,i}(\lambda)=h_\lambda(\xi_{0,i}(0))$. Let us set $E_\lambda\pe h_\lambda(E_0)$. Since $h_\lambda$ conjugates $f_0$ to $f_\lambda$ on $E_0$, for $\lambda\in\chi_i^{-1}\{0\}$ and $n\geq0$ we get
\begin{center}
$\xi_{n,i}(\lambda)=f_\lambda^n\big(\xi_{0,i}(\lambda)\big)=f_\lambda^n\circ h_\lambda(\xi_{0,i}(0))=h_\lambda(\xi_{n,i}(0))$.
\end{center}
The equicontinuity of $(\xi_{n,i})_{n\geq0}$ then follows from the uniform continuity of $h$ on $\mathbb{B}(0,r)\times E_0$. This proves $(1)$.

~

\par By $(1)$ we already know that if $c_i$ is active at $\lambda_0\in\chi_i^{-1}\{0\}$ then $\chi_i\not\equiv0$. We proceed by contradiction and assume that $\xi_{0,i}\not\equiv\nu_{0,i}$ on $\B(0,r)$ and that $c_i$ is passive at $\lambda_0\in\chi_i^{-1}\{0\}$. By remark \ref{Remvtholo} it suffices to consider the case where $\lambda_0=0$. Let $\mathcal{N}_\delta\subset\p^1$ be a $\delta$-neighborhood of $E_0$ and let $K>1$ be the hyperbolicity constant of $E_0$. Up to reducing $K$ and $r$ we may assume that 
\begin{eqnarray}
|f_\lambda'(z)|\geq K\text{ for every }\lambda\in\B(0,r)\text{ and every }z\in\mathcal{N}_\delta.
\label{constantehyperbolique}
\end{eqnarray}
By continuity of $h$ we may also assume that $E_\lambda\subset\mathcal{N}_\delta$ for every $\lambda\in\B(0,r)$. Since $(\xi_{n,i})_{n\geq0}$ is equicontinuous at $0$ we may again reduce $r$ so that
\begin{eqnarray}
\xi_{n,i}(\lambda)\in\D(\xi_{n,i}(0),\delta)\subset\mathcal{N}_\delta\textup{, for }\lambda\in\B(0,r)\textup{ and }n\geq0.
\label{equicontinuity}
\end{eqnarray}
For $n\geq0$ we set $\epsilon_{n,i}(\lambda)\pe \xi_{n,i}(\lambda)-h_\lambda(\xi_{n,i}(0))=\xi_{n,i}(\lambda)-\nu_{n,i}(\lambda)$ for every $\lambda\in\B(0,r)$. Notice that $\epsilon_{0,i}\equiv\chi_i$ on $\B(0,r)$. By assumption $(\epsilon_{n,i})_{n\geq0}$ is equicontinuous at $0$. On the other hand, for every $n\geq0$ and every $\lambda\in\B(0,r)$ we have
\begin{center}
$\epsilon_{n+1,i}(\lambda)=f_\lambda(\xi_{n,i}(\lambda))-f_\lambda(\nu_{n,i}(\lambda))$,
\end{center}
which after differentiation yields
\begin{center}
$D_\lambda\epsilon_{n+1,i}=f_\lambda'\circ\xi_{n,i}D_\lambda\epsilon_{n,i}-\big(f_\lambda'\circ\nu_{n,i}-f_\lambda'\circ\xi_{n,i}\big)D_\lambda\nu_{n,i}+D_\lambda f_\lambda\circ\xi_{n,i}-D_\lambda f_\lambda\circ\nu_{n,i}$.
\end{center}
Let $\varepsilon>0$ be small. Since $\xi_{n,i}$ and $\nu_{n,i}$ are equicontinuous, up to reducing $r$, the estimates \ref{constantehyperbolique} and \ref{equicontinuity} give
\begin{center}
$\|D_{\lambda} \epsilon_{n+1,i}\|\geq K\|D_{\lambda}\epsilon_{n,i}\|-\varepsilon$ 
\end{center}
By iterating, we get $\|D_{\lambda} \epsilon_{n,i}\|\geq K^n\big(\|D_{\lambda}\chi_i\|-\varepsilon(1-K^{-n})/(K-1)\big)$ which contradicts the equicontinuity of $(\epsilon_{n,i})_{n\geq0}$, since $\chi_i\not\equiv0$.\end{proof}

~

\begin{defi}
Let $(f_\lambda)_{\lambda\in\B(0,r)}$ be a holomorphic family of degree $d$ rational mpas with $k$ marked critical points $c_1,\ldots,c_k$. Assume that the set $\overline{\{f_0^n(c_j)/n\geq k_0, 1\leq j\leq k\}}$ is a compact $f_0$-hyperbolic set. The \emph{activity map} $\chi$ of $(f_\lambda)_{\lambda\in\B(0,r)}$ at $\lambda=0$ is defined by
\begin{eqnarray*}
\chi:\B(0,r) & \longrightarrow & \C^k\\
\lambda & \longmapsto & (\chi_1(\lambda),\ldots,\chi_k(\lambda))
\end{eqnarray*}
where the functions $\chi_i$ are given by $\chi_i(\lambda)\pe\xi_{0,i}(\lambda)-\nu_{0,i}(\lambda)$.
\end{defi}

~

\par The first useful fact we may remark is the following:

~

\begin{lm}
Let $(f_\lambda)_{\lambda\in\B(0,r)}$ be a holomorphic family of degree $d$ rational maps with $2d-2$ marked critical points and parametrized by a ball $\B(0,r)\subset\C^{2d-2}$. Assume that $f_0$ is $k$-Misiurewicz. Then, up to reducing $r$, the family $(f_\lambda)_{\lambda\in\chi^{-1}\{0\}}$ is stable and $f_\lambda$ is $k$-Misiurewicz for any $\lambda\in\chi^{-1}\{0\}$.
\label{lmstability}
\end{lm}

\begin{proof} Assume that $c_1(0),\ldots,c_k(0)\in\J_0$ and $c_{k+1}(0),\ldots,c_{2d-2}(0)\in\F_0$. Since $\chi=0$ on $\chi^{-1}\{0\}$, the critical points $c_1(\lambda),\ldots,c_k(\lambda)$ are captured by a hyperbolic set for any $\lambda\in\B(0,r)$ and Lemma \ref{lmactivite} states that they are passive on $\chi^{-1}\{0\}$. Moreover, by Proposition \ref{fatou}, the critical points $c_{k+1}(0),\ldots,c_{2d-2}(0)$ are in attracting basins of $f_0$ and then, up to reducing $r$, are passive and stay in the same basin.\end{proof}

~

\par Let $(f_\lambda)_{\lambda\in\B(0,r)}$ be a holomorphic family of degree $d$ rational maps. The group $\textup{Aut}(\p^1)$ acts on $\rat_d$ by conjugacy. The quotient space $\rat_d/\textup{Aut}(\p^1)$ is denoted by $\mathcal{M}_d$ and is called the \emph{moduli space} of degree $d$ rational maps. We denote by $\Pi:\B(0,r)\longrightarrow\mathcal{M}_d$ the natural projection. Then we have 
\begin{center}
$\Pi^{-1}(\Pi(\lambda_0))=\{\lambda\in\B(0,r)$ / $\exists \ \phi\in PSL_2(\C)$ s.t. $f_\lambda\circ\phi=\phi\circ f_{\lambda_0}$ on $\p^1\}$
\end{center}
for every $\lambda_0\in\B(0,r)$.

~

\par Our main goal in this section is to prove Theorem \ref{tmtransverseintro} we reformulate as follows:

~

\begin{tm}[Weak transversality]
Let $(f_\lambda)_{\lambda\in\B(0,r)}$ be a holomorphic family of degree $d$ rational maps with $2d-2$ marked critical points and parametrized by a ball $\B(0,r)\subset\C^{2d-2}$. Assume that $f_0$ is $k$-Misiurewicz but not a flexible Latt\`es map and that for every $\lambda_0\in \B(0,r)$ the set $\Pi^{-1}(\Pi(\lambda_0))$ is discrete. Let $\chi$ be the activity map of $(f_\lambda)_{\lambda\in\B(0,r)}$ at $\lambda=0$. Then $\textup{codim} \ \chi^{-1}\{0\}=k$. 
\label{tmtransversalite}
\end{tm}

\subsection{Good families} The notion of good family was introduced by Aspenberg (see \cite{Aspenberg1} pages 3-4), we begin with recalling it:

~

\begin{defi}
Let $(f_\lambda)_{\lambda\in\B(0,r)}$ be a holomorphic family of degree $d$ rational maps with $2d-2$ marked critical points and let $X\subset\B(0,r)$ be an analytic set containing $0$. We say that $(f_\lambda)_{\lambda\in X}$ is a \emph{good family} of rational maps if for every $\lambda_0\in X$ the set $\Pi^{-1}(\Pi(\lambda_0))$ is discrete and if every attracting cycle $\mathcal{C}_\lambda$ of $f_\lambda$ which depends holomorphically on $\lambda\in\B(0,r)$ satisfies the following assertions:
\begin{enumerate}
  \item the multiplier of $\mathcal{C}_\lambda$ is constant on $X$,
\item the multiplicity of the critical points of $f_\lambda$ which belong to the attracting basin $\mathcal{A}_\lambda$ of the cycle $\mathcal{C}_\lambda$ for every $\lambda\in\B(0,r)$ are constant on $X$,
  \item if $z(\lambda)\in\mathcal{C}_\lambda$ satisfies $f_\lambda^p(z(\lambda))=z(\lambda)$ for every $\lambda\in X$ and some integer $p\geq1$, there exists a Koenigs/B\"ottcher-coordinate $\varphi_\lambda$ for $f_\lambda^p$ at $z(\lambda)$ such that the following holds: let $c(\lambda)\in C(f_\lambda)$ lie in the attracting basin $\mathcal{A}_\lambda$ of the cycle $\mathcal{C}_\lambda$ for every $\lambda\in\B(0,r)$. Then, there esists an integer $n\geq1$ such that $\varphi_\lambda(f_\lambda^{n}(c(\lambda)))=\varphi_0(f_0^{n}(c(0)))$ for all $\lambda\in X$.
\end{enumerate}
We say that a good family $(f_\lambda)_{\lambda\in X}$ is \emph{smooth} if $X$ is smooth.
\end{defi}

~

\par The main usefulness of good families is revealed by the following Proposition:

~

\begin{prop}
Let $(f_\lambda)_{\lambda\in \B(0,r)}$ be a smooth good family of rational maps. Assume that $0$ is a stable parameter in the family $(f_\lambda)_{\lambda\in\B(0,r)}$. Assume also that $f_\lambda$ admits an attracting cycle $\mathcal{C}_\lambda$ holomorphically depending on $\lambda\in\B(0,r)$ and denote by $\mathcal{A}_\lambda$ its attracting basin. Then there exists $0<\rho\leq r/3$ and a unique holomorphic motion $\phi:\B(0,\rho)\times \p^1\longrightarrow\p^1$ such that for any $\lambda\in\B(0,\rho)$:
\begin{enumerate}
\item the following diagram commutes:
 \begin{center}
$\xymatrix {\relax
\p^1 \ar[r]^{\phi_{\lambda}}\ar[d]_{f_0} & \p^1 \ar[d]^{f_\lambda}\\
\p^1 \ar[r]_{\phi_{\lambda}} & \p^1, }$
\end{center}
\item $\phi_\lambda:\mathcal{A}_0\longrightarrow\mathcal{A}_\lambda$ is a biholomorphism.
\end{enumerate} 
\label{mvtfatou}
\end{prop}

\begin{proof} We adapt here the argument of the proof of Theorem 7.4 of \cite{mcmullensullivan}. As $0$ is a stable parameter, if $\{c_1(0),\ldots,c_k(0)\}=\J_0\cap C(f_0)$, then $\{c_1(\lambda),\ldots,c_k(\lambda)\}=\J_\lambda\cap C(f_\lambda)$ for $\lambda$ close enough to $0$ and there exists a holomorphic motion $\phi:\B(0,\rho_1)\times\J_0\longrightarrow\p^1$ which conjugates the dynamics on $\J_0$. Remark also that, by items $(2)$ and $(3)$ of the definition of good family, $\bigcup_{n\geq0}f_0^n(C(f_0)\cap\F_0)$ admits a dynamical holomorphic motion. By the $\lambda$-Lemma, it extends uniquely to a dynamical holomorphic motion of $\overline{\bigcup_{n\geq0}f_0^n(C(f_0))}\cup\J_0$.
\par Setting $\rho\pe\rho_1/3$, Bers-Royden extension Theorem gives a unique holomorphic motion $\tilde\phi:\B(0,\rho)\times \p^1\longrightarrow\p^1$ which extends $\phi$ and such that the Beltrami form $\mu_\lambda$ of $\tilde\phi_\lambda$ is harmonic (see \cite{BersRoyden} Theorem 3 and \cite{mcmullensullivan} page 381 for the definition of harmonic Beltrami forms). To conclude the proof of $(1)$, we thus have to prove that $f_\lambda\circ \tilde\phi_\lambda=\tilde\phi_\lambda\circ f_0$ on $\p^1$. The map
\begin{eqnarray*}
f_\lambda :\p^1\setminus\overline{\bigcup_{n\geq0}f_\lambda^n(C(f_\lambda))}\cup\J_\lambda & \longrightarrow & \p^1\setminus\overline{\bigcup_{n\geq0}f_\lambda^n(C(f_\lambda))}\cup\J_\lambda
\end{eqnarray*}
is a covering map. We define another motion by setting 
\begin{center}
$\psi_\lambda(z)\pe f_\lambda^{-1}\circ\tilde\phi_\lambda\circ f_0(z)$, for $z\in\p^1\setminus\overline{\bigcup_{n\geq0}f_\lambda^n(C(f_\lambda))}\cup\J_\lambda$,
\end{center}
where the inverse branch is choosen so that $\psi_0(z)=z$. As $\phi_\lambda$ conjugates the dynamics on $\overline{\bigcup_{n\geq0}f_0^n(C(f_0))}\cup\J_0$, the holomorphic motion $\psi$ is also an extension of $\phi$. Since $f_0$ and $f_\lambda$ are holomorphic, the Beltrami form of $\psi_\lambda$ is $f_0^*\mu_\lambda$. The map $f_0$ being a local isometry for the hyperbolic metric on $\p^1\setminus\overline{\bigcup_{n\geq0}f_0^n(C(f_0))}\cup\J_0$ and the form $\mu_\lambda$ being harmonic, the Beltrami form $f_0^*\mu_\lambda$ is also harmonic. By uniqueness of Bers-Royden extension, we get $\psi=\tilde\phi$, which proves $(1)$. We denote by $\phi$ this unique extension.
\par Let us now prove $(2)$. As $\phi_\lambda\circ f_0=f_\lambda\circ \phi_\lambda$ on $\p^1$, it is obvious that $\phi_\lambda(\mathcal{A}_0)=\mathcal{A}_\lambda$ and that $\phi_\lambda\circ f_0^q=f_\lambda^q\circ \phi_\lambda$ on $\p^1$, where $q$ is the period of the cycle $\mathcal{C}_\lambda$ whose attracting basin is $\mathcal{A}_\lambda$. Take now $z(\lambda)\in\mathcal{C}_\lambda$ and $\varphi_\lambda$ a local coordinate at $z(\lambda)$ which conjugates $f_\lambda^q$ to its normal form. It is easy to see that $\varphi_\lambda\circ\phi_\lambda=\varphi_0$. In a neighborhood of $z(\lambda)$, this yields:
\begin{center}
$\displaystyle\frac{\partial\phi_\lambda}{\partial\overline{z}}(z)=\left(\varphi_\lambda'\circ\phi_\lambda(z)\right)^{-1}\cdot\frac{\partial\varphi_0}{\partial\overline{z}}(z)\equiv0$,
\end{center}
i.e. $\phi_\lambda$ is holomorphic in a neighborhood of $z(\lambda)$. We thus can find $\Omega_0\subset\mathcal{A}_0$ an $f_0$-invariant open set on which $\phi_\lambda$ is holomorphic. For any $z\in f_
 0^{-1}(\Omega_0)\setminus\big(C(f_0)\cup f_0(C(f_0))\big)$, we have
\begin{center}
$\displaystyle\frac{\partial\phi_\lambda}{\partial\overline{z}}(z)=\big(f_\lambda'\circ\phi_\lambda(z)\big)^{-1}\cdot\frac{\partial\phi_\lambda}{\partial\overline{z}}\circ f_0(z)\cdot f_0'(z)=0$,
\end{center}
since $\phi_\lambda$ is holomorphic on $\Omega_0$. This means that $\phi_\lambda$ is holomorphic on $f_0^{-1}(\Omega_0)$. Since $\mathcal{A}_0=\bigcup_{n\geq1}f_0^{-n}(\Omega_0)$, the Proposition follows by induction.\end{proof}

~

\par Let us see how this Proposition implies the weak transversality in good families:

~

\begin{prop}
Let $(f_\lambda)_{\lambda\in X}$ be a good family of degree $d$ rational maps with $2d-2$ marked critical points and $0\in X$. Assume that $f_0$ is $k$-Misiurewicz but not a flexible Latt\`es map and $c_1(0),\ldots,c_k(0)\in\J_0$. Let $\chi$ be the activity map of $(f_\lambda)_{\lambda\in X}$ at $\lambda=0$.
\begin{enumerate}
\item Then there exists $1\leq i\leq k$ such that the critical point $c_i$ is active at $0$ in $X$, which means that $\chi^{-1}_i\{0\}$ has codimension $1$ in $X$.
\item Assume moreover that $\dim_\C X\geq k$. Then the analytic set $\chi^{-1}\{0\}$ has codimension $k$ in $X$.
\end{enumerate}
\label{lmactif}
\end{prop}

\begin{proof} Let us begin with proving $(1)$. We proceed by contradiction assuming that $c_i$ is passive in $X$ for every $1\leq i\leq k$, which means that $0$ is stable in $X$. Let $V_0\subset X$ be a neighborhood of $0$ in $X$ which, since $f_0$ is not a flexible Latt\`es map, may be assumed not to contain flexible Latt\`es maps. By Lemma \ref{lmactivite}, $\chi\equiv0$ on $V_0$ and, by Lemma \ref{lmstability}, $f_\lambda$ is $k$-Misiurewicz for all $\lambda\in V_0$. Let $\lambda_0\in {V_0}_{\textup{reg}}$ and $\psi:\D\longrightarrow {V_0}_{\textup{reg}}$ be a non-constant holomorphic disc centered at $\lambda_0$. In the following we will identify $\D$ and $\psi(\D)$ and therefore work with the family $(f_\lambda)_{\lambda\in\D}$.

~

\par If $1\leq k<2d-2$, then $\J_{\lambda_0}\neq\p^1$. Taking into account item $3$ of Proposition \ref{fatou}, Proposition \ref{mvtfatou} asserts that (up to reducing $\D$) there exists a holomorphic motion $\phi:\D\times\p^1\longrightarrow\p^1$ which conjugates $f_{\lambda_0}$ and $f_\lambda$ and which is holomorphic on $\mathcal{F}_{\lambda_0}$ for all $\lambda\in\D$. Item $4$ of Proposition \ref{fatou} ensures that $\phi_\lambda\in\textup{Aut}(\p^1)$. This contradicts our assumption that $\Pi^{-1}(\Pi(\lambda_0))$ is discrete.

~

\par If $k=2d-2$, then $\J_{\lambda_0}=\p^1$. As $\lambda_0$ is a stable parameter, by the Ma\~n\'e-Sad-Sullivan Theorem (see \cite{MSS} Theorem B) there exists a quasiconformal holomorphic motion $\phi:\D\times\p^1\longrightarrow\p^1$ which conjugates $f_{\lambda_0}$ to $f_\lambda$ on $\p^1$. For $\lambda\in\D\setminus\{\lambda_0\}$ we denote by $\mu^\lambda$ the Beltrami form satisfying:

\begin{center}
$\displaystyle \frac{\partial\phi_\lambda}{\partial\overline{z}}=\mu^\lambda\frac{\partial\phi_\lambda}{\partial z}$,
\end{center}
and by $S_\lambda$ its support. There exists a $\lambda_1\in\D\setminus\{\lambda_0\}$ such that $\textup{Leb}(S_{\lambda_1})>0$, otherwise $\phi_\lambda$ would belong to $\Aut(\p^ 1)$ and $\Pi^{-1}(\Pi(\lambda_0))$ could not be discrete. Since $f_{\lambda_1}$ carries an invariant line field on its Julia set, item $5$ of Proposition \ref{fatou} states that $f_{\lambda_1}$ is a flexible Latt\`es map. This is a contradiction since $V_0$ contains no flexible Latt\`es maps.

~

\par Let us now establish $(2)$. We have proved in $(1)$ that there exists $1\leq i\leq k$ such that $\textup{codim}\ \chi_{i}^{-1}\{0\}=1$. If $k\geq2$ then $\chi_{i}^{-1}\{0\}$ is a good family of dimension $k-1\geq1$ and assertion $(1)$ provides $j\neq i$ and $c_j$ is active at $0$ in $(f_\lambda)_{\lambda\in \chi_{i}^{-1}\{0\}}$. Iterating this $k-1$ times we obtain $(2)$.\end{proof}

\subsection{Ubiquity of good families and weak transversality}
We now prove Theorem \ref{tmtransversalite}. According to item $2$ of Proposition \ref{lmactif}, it suffices to exhibit a good family $(f_\lambda)_{\lambda\in X}$ such that $0\in X\subset\B(0,r)$ and $\dim_\C X\geq k$.

\par When $k=2d-2$, $f_{\lambda_0}$ has empty Fatou set and the family $X=\B(0,r)$ is obviously good. We may therefore suppose $1\leq k<2d-2$. As usual we will assume that $c_1(0),\ldots,c_k(0)\in\J_0$ and $c_{k+1}(0),\ldots,c_{2d-2}(0)\in\F_0$. \emph{We will be brought to reduce $r$ without mentionning it.}

~

\par Let us begin with finding a family $N\subset\B(0,r)$ of sufficiently high dimension and satisfying the assumptions $1$ and $2$ of the definition of good family. By Proposition \ref{fatou}, $c_{k+1}(0),\ldots,c_{2d-2}(0)$ are in attracting basins of $f_0$. Let us denote by $\mathcal{C}_1,\ldots,\mathcal{C}_p$ the attracting cycles of $f_0$, by $q_1,\ldots,q_p\geq1$ their respective periods and by $m_1,\ldots,m_p\in\D$ their respective multipliers. The set
\begin{center}
$N\pe\B(0,r)\cap \Per_{q_1}(m_1)\cap\cdots\cap\Per_{q_p}(m_p)$
\end{center}
is an analytic subset of $\B(0,r)$ of codimension at most $p$ containing $0$. Replacing $N$ by $N\cap\{\lambda\in\B(0,r)$ / $c_i(\lambda)=c_j(\lambda)\}$ when $c_i\not\equiv c_j$, $c_i(0)=c_j(0)$ and $i,j\geq k+1$ we get an analytic subset $N$ of $\B(0,r)$ satisfying items $1$ and $2$ of the definition of good families. Moreover the codimension of $N$ is at most:

\begin{center}
$p+\displaystyle\sum\big(\textup{mult}(c)-1\big)$
\end{center}
where the sum is taken over all critical points of $f_0$ contained in $\F_0$ and counted \emph{without} multiplicity.

~

\par To conclude we have to find a subfamily of $N$ satisfying item $3$ of the definition of good family with codimension at most $2d-2-k$ in $\B(0,r)$. Let $\mathcal{C}(0)$ be an attracting cycle of $f_0$ of period $p$ and let $z(0)\in\mathcal{C}(0)$ be such that $f_0^p(z(0))=z(0)$. By the Implicit Function Theorem $\mathcal{C}(0)$ and $z(0)$ may be followed holomorphically in $\B(0,r)$. Since $N$ satisfies item $1$ of the definition of good family we have two distinct cases:
\begin{itemize}
\item $\mathcal{C}(0)$ is not super-attracting. Then there exists a critical point $c_i(\lambda)$ converging to $\mathcal{C}(\lambda)$ whose orbit is infinite. Let $\varphi_\lambda:V_\lambda\longrightarrow\C$ be the Koenigs coordinate of $f_\lambda^p$ at $z(\lambda)$ and $n_i\geq1$ be large enough so that $f_\lambda^{n_i}(c_i(\lambda))\in V_\lambda$ for every $\lambda\in \B(0,r)$. We normalize $\varphi_\lambda$ by setting:
\begin{center}
$\displaystyle\psi_\lambda(z)\pe\frac{\varphi_\lambda(z)}{\varphi_\lambda(f_\lambda^{n_i}(c_i(\lambda)))}$
\end{center}
for every $\lambda\in N$ and every $z\in  V_\lambda$.
\item $\mathcal{C}(0)$ is super-attracting. There exists a critical point $c_i(\lambda)$ of $f_\lambda$ and $0\leq n_i\leq p-1$ such that $f_\lambda^{n_i}(c_i(\lambda))=z(\lambda)$ and $\varphi_\lambda(f_\lambda^{n_i}(c_i(\lambda)))\equiv0$ if $\varphi_\lambda$ is the B\"ottcher coordinate of $f_\lambda^p$ at $z(\lambda)$. In that case we set $\psi_\lambda\pe\varphi_\lambda$.
\end{itemize}

~

\par If $c_j(0)$ is any other critical point contained in the attracting basin of $\mathcal{C}(0)$, then we replace $N$ by
\begin{center}
$N\cap\{\lambda\in\B(0,r)$ / $\psi_\lambda\big(f_\lambda^{n_j}(c_j(\lambda))\big)=\psi_0\big(f_0^{n_j}(c_j(0))\big)\}$.
\end{center}
Since, by definition of $\psi_\lambda$, this operation does not change $N$ for $p$ distinct critical points of $f_0$, the codimension of $N$ is now at most equal to

\begin{center}
$p+\displaystyle\sum\big(\textup{mult}(c)-1\big)-p+\sum1=\sum\textup{mult}(c)=2d-2-k$,
\end{center}
where the sums are taken over critical points of $f_0$ contained in $\F_0$ and counted \emph{without} multiplicity. This ends the proof of Theorem \ref{tmtransversalite}.

\subsection{An application} Let $(f_\lambda)_{\lambda\in X}$ be a holomorphic family of degree $d$ rational maps with a marked critical point $c$, which is active at some $\lambda_0\in X$. Let $E\subset\J_{\lambda_0}$ be a $f_{\lambda_0}$-hyperbolic set containing at least three points and $h_\lambda$ its dynamical holomorphic motion. As it is well-known, Montel's Theorem allows one to find $\lambda_1\in X$, arbitrarily close to $\lambda_0$, and $n_1\geq1$ such that $f_{\lambda_1}^{n_1}(c(\lambda_1))\in h_{\lambda_1}(E)$. However, there is in general no way to extend this to the case of two, or more, active critical points. Weak transversality allows such a generalization when $f_{\lambda_0}$ is Misiurewicz. This will be used in Section \ref{sectiondim}.

~

\begin{lm}
Let $(f_\lambda)_{\lambda\in\B(0,r)}$ be a holomorphic family of degree $d$ rational maps with $2d-2$ marked critical points. Assume that for every $\lambda_0\in\B(0,r)$, the set $\Pi^{-1}(\Pi(\lambda_0))$ is discrete and that $f_0$ is $k$-Misiurewicz but not a flexible Latt\`es map. Let also $E\subset\J_{0}$ be a compact $f_{0}$-hyperbolic set containing at least three points. Then for any $0<\varepsilon<\rho$, there exists $\lambda_0\in \B(0,\varepsilon)$ and an integer $N_0\geq1$ such that $f_{\lambda_0}$ is $k$-Misiurewicz and $f_{\lambda_0}^{N_0}(c_j(\lambda_0))\in h_{\lambda_0}(E)$ for $1\leq j\leq k$.\label{lmapproxmis}
\end{lm}

\begin{proof} Let $E_1$ and $E_2$ be two compact $f_0$-hyperbolic sets. Assume that $E_2$ contains at least three distinct points and let $h:\B(0,\rho)\times (E_1\cup E_2)\longrightarrow\p^1$ be its dynamical holomorphic motion. Suppose that
\begin{center}
$f_0^{k_0}(c_1(0)),\ldots,f_0^{k_0}(c_p(0))\in E_1$ and $f_0^{k_0}(c_{p+1}(0)),\ldots, f_0^{k_0}(c_k(0))\in E_2$.
\end{center}
We only need to find a parameter $\lambda\in\B(0,\varepsilon)$ and an integer $N_0\geq k_0$ such that
\begin{center}
$f_{\lambda_0}^{N_0}(c_1(\lambda_0)),\ldots,f_{\lambda_0}^{N_0}(c_{p-1}(\lambda_0))\in h_{\lambda_0}(E_1)$ and $f_{\lambda_0}^{N_0}(c_p(\lambda_0)),\ldots, f_{\lambda_0}^{N_0}(c_k(\lambda_0))\in h_{\lambda_0}(E_2)$.
\end{center}
\par Let $\varepsilon>0$ and denote by $\chi:\B(0,\rho)\longrightarrow\C^k$ the activity map for $(f_\lambda)_{\lambda\in\B(0,r)}$ at $0$. By Theorem \ref{tmtransverseintro} and Lemma \ref{lmactif} the critical point $c_p$ is active at $0$ in the family $X_0\pe\chi^{-1}_1\{0\}\cap\cdots\cap\chi^{-1}_{p-1}\{0\}\cap\chi^{-1}_{p+1}\{0\}\cap\cdots\cap\chi^{-1}_k\{0\}$. By Montel's Theorem, there exists $\lambda_0\in X\cap\B(0,\varepsilon)$ and $N_0\geq k_0$ such that $f_{\lambda_0}^{N_0}(c_p(\lambda_0))\in h_{\lambda_0}(E_2)$. Since $\lambda_0\in X_0$ and $N_0\geq k_0$ we get $f_{\lambda_0}^{N_0}(c_j(\lambda_0))\in h_{\lambda_0}(E_1)$ if $1\leq j\leq p-1$ and $f_{\lambda_0}^{N_0}(c_j(\lambda_0))\in h_{\lambda_0}(E_2)$ if $p+1\leq j\leq k$. After item (3) of Proposition \ref{fatou}, the $2d-2-k$ remaining critical points $f_{\lambda_0}$ belong to attracting basins of $f_{\lambda_0}$. Since this condition is stable under small enough perturbations, the map $f_{\lambda_0}$ is $k$-
 Misiurewicz.\end{proof}

\section{Local Hausdorff dimension estimates in $\rat_d$}\label{sectiondim}

We consider here the sets $\mathfrak{M}_k$ defined by:

~

\begin{center}
$\mathfrak{M}_k\pe\{f\in\rat_d$ / $f$ is $k$-Misiurewicz but not a flexible Latt\`es map\\ and $f$ has simple critical points$\}$.
\end{center} 

 \begin{rem}
The condition "$f$ is not a flexible Latt\`es map" is superfluous when $d$ is not the square of an integer or $k<2d-2$.
\end{rem}

~

\par Our goal here is to prove Theorem \ref{tm2intro}. Following Shishikura \cite{Shishikura2}, we define the \emph{hyperbolic dimension} of $f\in\rat_d$ by:
\begin{center}
$\dim_{\textup{hyp}}(f)\pe\sup\{\dim_H(E)$ / $E$ is compact, $f$-hyperbolic and homogeneous$\}$.
\end{center}
Recall that a subset $E$ of a metric space $(X,d)$ is \emph{homogeneous} if, for every open set $U\subset X$ such that $U\cap E\neq\emptyset$, we have $\dim_H(U\cap E)=\dim_H(E)$.

\subsection{Technical preliminaries}
The proof requires weak transversality in $\rat_d$. Recall that $\Pi:\rat_d\longrightarrow\mathcal{M}_d$ is the quotient map. For $f\in\rat_d$ we denote by $\textup{Aut}(f)$ the stabilizer of $f$ under the action of $\textup{Aut}(\p^1)$, i.e. $\textup{Aut}(f)=\{\phi\in PSL_2(\C)$ / $\phi^{-1}\circ f\circ \phi=f\}$. 

~

\begin{lm}
Let $1\leq k\leq 2d-2$ and $f_0\in\mathfrak{M}_k$. Let $\B(0,r)$ be a ball centered at $f_0$ small enough so that $(f_\lambda)_{\lambda\in\B(0,r)}$ is with $2d-2$ marked critical points. Let also $\chi$ be the activity map of $(f_\lambda)_{\lambda\in\B(0,r)}$ at $\lambda=0$. Then $\textup{codim}\ \chi^{-1}\{0\}=k$.
\label{lmvariete}
\end{lm}

\begin{proof} One can prove that for every $f\in\rat_d$ there exists a local $(2d-2)$-dimensional submanifold $V_f$ of $\rat_d$ containing $f$ which is transversal to the orbit $\mathcal{O}(f)\subset\rat_d$ of $f$ under the action of $\textup{Aut}(\p^1)$ (see for example \cite{BB1} page 226). The stabilizer $\textup{Aut}(f)$ of $f$ is a finite group and $V_f$ is invariant by the action of $\textup{Aut}(f)$. Moreover, $\Pi(V_f)$ is an open subset of $\mathcal{M}_d$ and $\Pi$ induces a biholomorphism $V_f/\textup{Aut}(f)\longrightarrow\Pi(V_f)$.
\par Thus for every $g\in V_f$, the set $\Pi^{-1}(\Pi(g))$ contains at most $\textup{card}(\textup{Aut}(f))<+\infty$ elements. We conclude by applying Theorem \ref{tmtransverseintro} in $\B(0,r)\cap V_f$.\end{proof}

~

\par To control how the Hausdorff dimension of a compact set is affected by holomorphic motion, we shall use the following precise H\"{o}lder-regularity statement. A proof is provided in \cite{dujardin} Lemma 1.1.

~

\begin{lm}
Let $E$ be a compact subset of $\p^1$ and $h:\B(0,r)\times E\longrightarrow\p^1$ be a holomorphic motion. Then there exists $0<r_1\leq r$ and $\eta>0$ such that:
\begin{eqnarray}
\frac{1}{C(r')}|z-z'|^{\frac{r_1+\|\lambda\|}{r_1-\|\lambda\|}}\leq|h(\lambda,z)-h(\lambda,z')|\leq C(r')|z-z'|^{\frac{r_1-\|\lambda\|}{r_1+\|\lambda\|}}
\label{biholder}
\end{eqnarray}
for $0<r'<r_1$, all $z_0\in E$, all $z,z'\in \D(z_0,\eta)\cap E$ and all $\lambda\in\B(0,r')$.
\label{lmbiholder}
\end{lm}

~

\par It will be crucial to observe that, after a small perturbation, all active critical points of a Misiurewicz map may be assumed to be captured by some hyperbolic set of "big" Hausdorff dimension.

~

\begin{lm}
Let $0<\varepsilon<1$ and $f\in\mathfrak{M}_k$. Let $V_0\subset\rat_d$ be a neighborhood of $f$ and $F_0\subset\J_f$ be a compact homogeneous $f$-hyperbolic set such that $\dim_H(F_0)\geq\dim_{\textup{hyp}}(f)-\varepsilon$. Then we may find $f_t\in\mathfrak{M}_k$ arbitrary close to $f$ and a compact $f_t$-hyperbolic set $F_t$ such that $f_t^N(C(f_t)\cap\J_{f_t})\subset F_t$ for some $N\geq1$ and $\dim_H(U_t\cap F_t)\geq\dim_{\textup{hyp}}(f)-2\varepsilon$ for all open set $U_t$ intersecting $F_t$.
\label{lmdimpostcritic}
\end{lm}

\begin{proof} We may assume that $f=f_0$ where $(f_\lambda)_{\lambda\in\B(0,r)}$ is a holomorphic family and $f_\lambda\in V_0$ for all $\lambda\in\B(0,r)$. Let $h:\B(0,r)\times F_0\cup P^{k_0}(f)\longrightarrow\p^1$ be the dynamical holomorphic motion and $F_\lambda\pe h_\lambda(F_0)$. As it follows from Lemma \ref{lmbiholder} and from the homogeneity of $F_0$, there exists $0<r_1\leq r$ such that:
\begin{eqnarray}
& \displaystyle\dim_H\big(F_{\lambda}\cap U_\lambda\big)\geq\frac{r_1-\|\lambda\|}{r_1+\|\lambda\|}\dim_H\big(F_0\big) &
\label{dimbiholder}
\end{eqnarray}
for all $\lambda\in\B(0,r_1/2)$ and all open set $U_\lambda\subset\p^1$ intersecting $F_\lambda$.
\par Suppose that $\J_{0}\cap C(f_{0})=\{c_1(0),\ldots,c_k(0)\}$. Let $0<\tau<1/2$. According to Lemma \ref{lmapproxmis}, there exists $\lambda_0\in\B(0,\tau r_1)$ and $N\geq1$ satisfying $f_{\lambda_0}^N(c_j(\lambda_0))\in F_{\lambda_0}=h_{\lambda_0}(F_0)$ for any $1\leq j\leq k$. By \ref{dimbiholder} we have $\dim_H\big(F_{\lambda_0}\cap U_{\lambda_0}\big)\geq \frac{1-\tau}{1+\tau}\dim_H\big(F_0\big)\geq \frac{1-\tau}{1+\tau}\big(\dim_{\textup{hyp}}(f)-\varepsilon\big)$ and we conclude by taking $\tau$ small enough.\end{proof}

~

\par Let us now describe briefly the mechanism of the proof of Theorem \ref{tm2intro}. Let $E_f$ be a compact homogeneous $f$-hyperbolic set with $\dim_H(E_f)\geq\dim_{\textup{hyp}}(f)-\varepsilon$. Using Lemma \ref{lmdimpostcritic} we may find $g\in\mathfrak{M}_k$, arbitrary close to $f$, and $N\geq1$ such that $g^N(C(g)\cap\J_g)\subset E_g$ and $\dim_H(E_g\cap\D(g^N(c_j(g)),\eta))\geq\dim_{\textup{hyp}}(f)-2\varepsilon$. Using holomorphic motions we will build a transfer map which "copies" small pieces of $E_g^k\times\C^{2d+1-k}$ into a small neighborhood of $g$ in $\mathfrak{M}_k$ and enjoys good H\"{o}lder regularity properties. Finally, we precisely estimate the Hausdorff dimension of these copies to get a lower bound for the local Hausdorff dimension of $\mathfrak{M}_k$ near $f$.

\subsection{The transfer map}
In this subsection, we define a transfer map and investigate its regularity properties in the following setting: let $g\in\mathfrak{M}_k$ and $V_0$ be a neighborhood of $g$ in $\rat_d$. We may assume that $g=g_0$ where $(g_\lambda)_{\lambda\in\B(0,r)}$ is a holomorphic family and $g_\lambda\in V_0$ for all $\lambda\in\B(0,r)$ with $\dim_\C\B(0,r)=2d+1$. Assume that $g^N_0(c_j(0))\in E_0$, where $E_0$ is a $g_0$-hyperbolic compact set, for $j\leq k$. Set $z_0\pe\big(g_0^N(c_1(0)),\ldots,g_0^N(c_k(0))\big)$ and let $h:\B(0,r)\times E_0\longrightarrow\p^1$ be the dynamical holomorphic motion of $E_0$. By Bers-Royden's Theorem (see \cite{BersRoyden} Theorem 3), the motion $h$ extends to a holomorphic motion $h:\B(0,\rho)\times \p^1 \longrightarrow \p^1$, with $\rho\pe r/3$. Denote by $\D_\eta^k(z_0)\subset(\p^1)^k$ the polydisc of radius $\eta$ centrered at $z_0$. Let $\eta_0>0$ and set:
\begin{eqnarray*}
\mathcal{X}: \B(0,\rho)\times\D_{\eta_0}^k(z_0) & \longrightarrow & \C^k\\
(\lambda,z) & \longmapsto & \big(\mathcal{X}_1(\lambda,z),\ldots,\mathcal{X}_k(\lambda,z)\big),
\end{eqnarray*}
where $\mathcal{X}_j(\lambda,z)\pe g^N_\lambda(c_j(\lambda))-h_\lambda(z_i)$. The map $\mathcal{X}$ defined above is well-defined in an appropriate system of coordinate on $\D_{\eta_0}^k(z_0)$ and continuous. Moreover, $\mathcal{X}(\lambda,\cdot)$ is injective on $\D_{\eta_0}^k(z_0)$ for every $\lambda\in\B(0,\rho)$ and $\mathcal{X}(\cdot,z)$ is holomorphic on $\B(0,\rho)$ for every $z\in\D_{\eta_0}^k(z_0)$.
 
~

\par In what follows we will write $\lambda\in\C^k\times\C^{2d+1-k}$ under the form $\lambda=\lambda'+\lambda''$ with $\lambda'\in\C^k$ and $\lambda''\in\C^{2d+1-k}$. Similarly we write $\B'(0',r')$ (resp. $\B''(0'',r'')$) the ball of $\C^k$ (resp. $\C^{2d+1-k}$) of radius $r'$ (resp. $r''$) centered at the origin.

~

\begin{propdefi}
Under the above assumptions, up to reducing $\rho$, there exists $\eta_1>0$, $0<\delta_1''\leq\delta_1'\leq\rho$ and a continuous map
\begin{eqnarray*}
\mathcal{T}:\D_{\eta_1}^k(z_0)\times\B''(0'',\delta_1'') & \longrightarrow & \B'(0',\delta_1')
\end{eqnarray*}
such that $\mathcal{X}\big(\mathcal{T}(z,\lambda'')+\lambda'',z\big)=0$ and satisfying the following properties: for any $0<\varepsilon<1$ there exists $0<\delta''\leq\delta'\leq\varepsilon\rho$ such that the following occurs:
\begin{enumerate}
\item there exists a constant $C_\varepsilon>0$ such that:
\begin{eqnarray*}
\|\mathcal{T}(z_1,\lambda'')-\mathcal{T}(z_2,\lambda'')\|\geq C_\varepsilon\|z_1-z_2\|^{\frac{1+\varepsilon}{1-\varepsilon}}
\end{eqnarray*}
for every $\lambda''\in\B''(0'',\delta'')$ and every $z_1,z_2\in\D_{\eta_1}^k(z_0)$,
\item $\mathcal{T}\Big(\D_{\eta_1}^k(z_0)\cap\big(E_0\big)^k\times\{\lambda''\}\Big)+\lambda''\subset\mathfrak{M}_k\cap\B(0,\varepsilon\rho_1)$ for any $\lambda''\in\B''(0'',\delta'')$.
\end{enumerate}
The map $\mathcal{T}$ is called the \emph{transfer map} associated to $g$.
\label{defitransfer}
\end{propdefi}

\begin{proof} Let $\chi$ be the activity map of $(g_\lambda)_{\lambda\in\B(0,r)}$ at $\lambda=0$ (see Section \ref{sectionactivitymap}), then $\mathcal{X}(\lambda,z_0)=\chi(\lambda)$. Therefore, by Lemma \ref{lmvariete}, the analytic set $\{\mathcal{X}(\lambda,z_0)=0\}$ has pure codimension $k$. Up to a linear change of coordinate and after reducing $\rho$ we may assume that $\C^{2d+1}=\C^k\times\C^{2d+1-k}$ and
\begin{center}
$\{\lambda'\in\B'(0',\rho)$ / $\mathcal{X}(\lambda',z_0)=0\}=\{0'\}$.
\end{center}
Let $p$ be the multiplicity of $0'$ as a zero of $\mathcal{X}(\cdot,z_0)$ in $\B'(0',\rho)$. Let $0<\delta_1'\leq r$ be such that $|\mathcal{X}(\lambda',z_0)|\neq0$ on $\partial\B'(0',\delta_1')$. Then 

\begin{center}
$0=|\mathcal{X}(\lambda',z_0)-\mathcal{X}(\lambda',z_0)|<|\mathcal{X}(\lambda',z_0)|$
\end{center}
for any $\lambda_1'\in\partial\B'(0',\delta_1')$ and, by continuity, there exists $0<\delta_1''\leq\delta_1'$ and $\eta_1>0$ such that 
\begin{center}
$|\mathcal{X}(\lambda'+\lambda'',z)-\mathcal{X}(\lambda',z_0)|-|\mathcal{X}(\lambda',z_0)|<0$
\end{center}
for any $\lambda''\in\B''(0'',\delta_1'')$, $z\in\D_{\eta_1}^k(z_0)$ and $\B'(0',\delta_1')\times\B''(0'',\delta_1'')\subset\B(0,\rho)$. A Rouch\'e-like Theorem then states that the map $\mathcal{X}(\cdot+\lambda'',z)$ admits exactly $p$ zeros counted with multiplicity in $\B'(0',\delta'_1)$ for every $\lambda''\in\B''(0'',\delta_1'')$ and every $z\in\D_{\eta_1}^k(z_0)$ (see \cite{Chirka}, Theorem 1, section 10.3 page 110). To define $\mathcal{T}(z,\lambda'')$ it thus suffices to pick one of the $p$ elements of $\{\mathcal{X}(\cdot+\lambda'',z)=0\}$. Since the set $\{(\lambda',\lambda'',z)\in\B'(0',\delta_1')\times\B''(0'',\delta_1'')\times\D_{\eta_1}^k(z_0) \ / \ \mathcal{X}(\lambda'+\lambda'',z)=0\}$ is an analytic set, one can choose $(\lambda'',z)\mapsto \mathcal{T}(z,\lambda'')$ continuous.

~

\par Let us establish $(1)$. According to Lemma \ref{lmbiholder}, up to reducing $\eta_1$ and $\rho$, for any $0<\varepsilon<1$, there exists constants $C_\varepsilon,C_\varepsilon'>0$ such that for $\lambda\in\B(0,\varepsilon\rho)$ and $z_1,z_2\in\D_{\eta_1}^k(z_0)$ we have:
\begin{eqnarray*}
C_\varepsilon'\|z_1-z_2\|^{\frac{\rho+\|\lambda\|}{\rho-\|\lambda\|}}\leq\|\mathcal{X}(\lambda,z_1)-\mathcal{X}(\lambda,z_2)\|\leq C_\varepsilon\|z_1-z_2\|^{\frac{\rho-\|\lambda\|}{\rho+\|\lambda\|}}.
\end{eqnarray*}
Moreover, as $\|z_1-z_2\|<1$,
\begin{eqnarray}
 C_\varepsilon'\|z_1-z_2\|^{\frac{1+\varepsilon}{1-\varepsilon}}\leq\|\mathcal{X}(\lambda,z_1)-\mathcal{X}(\lambda,z_2)\| \leq C_\varepsilon\|z_1-z_2\|^{\frac{1-\varepsilon}{1+\varepsilon}}.
\label{inegholderepsilon}
\end{eqnarray}
Since $\mathcal{X}$ is continuous in both variables and holomorphic in $\lambda$ on $\B(0,\rho)$, there exists $\tilde C_{\varepsilon}>0$ such that 
\begin{eqnarray}
\|\mathcal{X}(\lambda_1,z)-\mathcal{X}(\lambda_2,z)\|\leq \tilde C_{\varepsilon}\|\lambda_1-\lambda_2\|
\label{inegholo}
\end{eqnarray}
for every $\lambda_1,\lambda_2\in\B(0,\varepsilon\rho)$ and every $z\in\D_{\eta_1}^k(z_0)$. Combining \ref{inegholderepsilon} and \ref{inegholo} we find:
\begin{eqnarray*}
\|\mathcal{X}(\lambda_1,z_1)-\mathcal{X}(\lambda_2,z_2)\| \geq C_\varepsilon'\|z_1-z_2\|^{\frac{1+\varepsilon}{1-\varepsilon}}-\tilde C_{\varepsilon}\|\lambda_1-\lambda_2\|
\end{eqnarray*}
for every $\lambda_1,\lambda_2\in\B(0,\varepsilon\rho_1)$ and every $z_1,z_2\in\D_{\eta_1}^k(z_0)$. As $\mathcal{T}$ is continuous, up to reducing $\eta_1$, one can find $0<\delta''\leq \delta'\leq\varepsilon\rho_1$ such that $\mathcal{T}(\D_\eta^k(z_0)\times\B''(0'',\delta''))\subset\B'(0',\delta')\subset\B'(0',\varepsilon\rho)$. Taking $\lambda_1=\mathcal{T}(z_1,\lambda'')+\lambda''$ and $\lambda_2=\mathcal{T}(z_2,\lambda'')+\lambda''$ yields the announced estimate. The above construction and Lemma \ref{lmstability} directly give $(2)$.\end{proof}

\subsection{Local Hausdorff dimension estimates} We are now ready to end the proof of Theorem \ref{tm2intro}.  Let $f\in\mathfrak{M}_k$ and $V_0$ be a neighborhood of $f$ in $\rat_d$. Let $0<\varepsilon<1$ and $F_0$ be a homogeneous compact $f$-hyperbolic set such that $\dim_H(F_0)\geq\dim_{\textup{hyp}}(f)-\varepsilon$. Assume that $f=f_0$ where $(f_\lambda)_{\lambda\in\B(0,R)}$ is a holomorphic family and $f_\lambda\in V_0$ for all $\lambda\in\B(0,R)$, with $\dim\B(0,R)=2d+1$.
\par By Lemma \ref{lmdimpostcritic}, there exists $\lambda_0\in\B(0,R/2)$ and a compact $f_{\lambda_0}$-hyperbolic set $F_{\lambda_0}$ such that $f_{\lambda_0}^N(C(f_{\lambda_0})\cap\J_{f_{\lambda_0}})\subset F_{\lambda_0}$ for some $N\geq1$ and such that $\dim_H(U_j\cap F_{\lambda_0})\geq\dim_{\textup{hyp}}(f)-2\varepsilon$ for any neighborhood $U_j$ of $f_{\lambda_0}^N(c_j(\lambda_0))$ for $1\leq j\leq k$.

~

\par Let $0<r\leq R/2$ and set $E_0\pe F_{\lambda_0}$, $g_0\pe f_{\lambda_0}$ and $g_\lambda\pe f_{\lambda+\lambda_0}$. Let us now denote by $\mathcal{T}$ the transfer map associated to $g$ in the family $(g_\lambda)_{\lambda\in\B(0,r)}$. Let $\rho\leq r$, $\eta_1$, $\delta''\leq\delta'\leq\varepsilon\rho$ be given by Proposition \ref{defitransfer}. For $\lambda''\in\B''(0'',\delta'')$ we set:
\begin{center}
$\mathcal{E}_{\lambda''}\pe \mathcal{T}\Big(\D_\eta^k(z_0)\cap\big(E_0\big)^k\times\{\lambda''\}\Big)+\lambda''$.
\end{center}
From item $(1)$ of Proposition \ref{defitransfer} we get:
\begin{eqnarray*}
\dim_H\big(\mathcal{E}_{\lambda''}\big)\geq \frac{1-\varepsilon}{1+\varepsilon}\dim_H\big(\big(E_0\big)^k\cap\D_\eta^k(z_0)\big) \geq \frac{1-\varepsilon}{1+\varepsilon}\sum_{i=1}^k\dim_H\big(E_0\cap\D(z_{0,i},\eta)\big).
\end{eqnarray*}
By the choice of $E_0$, we then get
\begin{eqnarray}
\dim_H\big(\mathcal{E}_{\lambda''}\big)\geq\frac{1-\varepsilon}{1+\varepsilon}k\big(\dim_{\textup{hyp}}(f)-2\varepsilon\big).
\label{inegdimepsilon}
\end{eqnarray}
Set $\mathcal{E}\pe\{\lambda=\lambda'+\lambda''\in\B'(0',\delta')\times\B''(0'',\delta'')$ / $\lambda'\in \mathcal{E}_{\lambda''}\}\subset \B'(0',\delta')\times\B''(0'',\delta'')$. We shall now use the following Lemma (see \cite{McMullen3} Lemma 5.1 page 15):

~

\begin{lm}
Let $Y$ be a metric space and $X\subset Y\times[0,1]^k$. Denote by $X_t$ the slice $X_t\pe\{y\in Y$ / $(y,t)\in X\}$. If $X_t\neq\emptyset$ for almost every $t\in[0,1]^k$, then
\begin{center}
$\dim_H\big(X\big)\geq k+\dim_H\big(X_t\big)$, for almost every $t$.
\end{center}
\label{lmdimHMcMullen}
\end{lm}

~

\par Lemma \ref{lmdimHMcMullen} states that for almost every $\lambda''\in\B''(0'',\delta'')$
\begin{eqnarray}
\dim_H(\mathcal{E})\geq \dim_H(\B''(0'',\delta''))+\dim_H\big(\mathcal{E}_{\lambda''}\big).
\label{inegdimepsilon2}
\end{eqnarray}
Since $\dim_H\big(\B''(0'',\delta'')\big)=2(2d+1-k)$, \ref{inegdimepsilon} and \ref{inegdimepsilon2} give:
\begin{center}
$\dim_H(\mathcal{E})\geq 2(2d+1-k) +\displaystyle\frac{1-\varepsilon}{1+\varepsilon} k\big(\dim_{\textup{hyp}}(f)-2\varepsilon\big)$.
\end{center}
As $\mathcal{E}\subset\mathfrak{M}_k\cap V_0$ (see item $(2)$ of Proposition \ref{defitransfer}) we find:
\begin{center}
$\dim_H\big(\mathfrak{M}_k\cap V_0\big)\geq\displaystyle 2(2d+1-k)+\frac{1-\varepsilon}{1+\varepsilon} k\big(\dim_{\textup{hyp}}(f)-2\varepsilon\big)$.
\end{center}
We conclude by letting $\varepsilon$ tend to $0$.

\subsection{The Hausdorff dimension of $\mathfrak{M}_k$}
Here we give the Hausdorff dimension of $\mathfrak{M}_k$:

~

\begin{tm}
For any $1\leq k\leq 2d-2$ the set $\mathfrak{M}_k$ has Hausdorff dimension $2(2d+1)$ and is homogeneous.
\label{tmdimH2}
\end{tm}

\par When $k=1$, this has already been proved by Aspenberg and Graczyk in \cite{aspenberggraczyk}. When $k<2d-2$, the proof of this result is based on a finer perturbation Lemma than Lemma \ref{lmdimpostcritic} which can be established with parabolic implosion techniques:

~

\begin{lm}
Let $0<\epsilon<1$ and $f\in\mathfrak{M}_k$. Let $V_0\subset\rat_d$ be a neighborhood of $f$. Then we may find $f_t$ arbitrary close to $f$ and a compact $f_t$-hyperbolic set $E_t$ such that $f_t^N(C(f_t)\cap\J_{f_t})\subset E_t$ for some $N\geq1$ and $\dim_H(U_t\cap E_t)\geq2-2\epsilon$ for all open set $U_t$ intersecting $E_t$.
\label{parabolicimplosion}
\end{lm}

\begin{proof} Take $(f_\lambda)_{\lambda\in\B(0,r)}$ a holomorphic family of degree $d$ rational maps parametrized by a ball $\B(0,r)\subset\C^{2d+1}$, with $f=f_0$ and $f_\lambda\in V_0$ for all $\lambda\in\B(0,r)$. Recall that we denoted by $\chi(\lambda)=(\chi_1(\lambda),\ldots,\chi_k(\lambda))$ the activity map of $f_0$.
\par Lemma \ref{lmactif}, combined with Lemma \ref{lmvariete}, asserts that the critical point $c_1$ is active at $0$ in $X\pe \chi^{-1}_2\{0\}\cap\cdots\cap\chi^{-1}_k\{0\}\cap\B(0,r)$. Ma\~n\'e-Sad-Sullivan's Theorem gives $\lambda_1\in X$, arbitrary close to $0$, such that $f_{\lambda_1}$ has a non-persistent parabolic cycle. Since every critical point except $c_1$ is passive in $X$, its immediate parabolic basin contains exactly one critical point which must be $c_1(\lambda_1)$. In these precise conditions, Shishikura proved, using fine parabolic implosion techniques (see \cite{Shishikura2} Theorem 2 and \cite{TanLei} Theorem 1.1), that there exists $\lambda_2\in X$, arbitrary close to $\lambda_1$, 
 such that $\dim_{\textup{hyp}}(f_{\lambda_2})>2-\epsilon$ and $f_{\lambda_2}$ has a neutral cycle which is non-persistent in $X$.
\par Let now $E\subset\J_{\lambda_2}$ be a compact homogeneous $f_{\lambda_2}$-hyperbolic set with $\dim_H(E)\geq 2-\epsilon$ and $h:\B(\lambda_2,\rho)\times E\longrightarrow\p^1$ be the dynamical holomorphic motion of $E$ ($\rho\leq r-\|\lambda_2\|$). By Lemma \ref{lmbiholder} we may find $0<\rho'<\rho$ such that
\begin{center}
$\dim_H\big(h_\lambda(E)\cap U_\lambda\big)\geq\dim_H(E)-\epsilon$
\end{center}
for every $\lambda\in\B(\lambda_2,\rho')$ and every open set $U_\lambda$ intersecting $h_\lambda(E)$. Since $f_{\lambda_2}$ has a non-persistent neutral cycle, the parameter $\lambda_2$ is in the bifurcation locus of $X\cap\B(\lambda_2,\rho')$ and the critical point $c_1$ is active at $\lambda_2$. By Montel's Theorem there exists $\lambda_3\in X\cap\B(\lambda_2,\rho')$, arbitrary close to $\lambda_2$, and $N\geq1$ such that $f_{\lambda_3}^N(c_1(\lambda_3))\in h_{\lambda_3}(E)$ and $f_{\lambda_3}$ is $k$-Misiurewicz. To conclude we now proceed as in the proof of Lemma \ref{lmdimpostcritic}.\end{proof}

\textit{Proof of Theorem \ref{tmdimH2}.}  By Theorem 1.1 of \cite{McMullen2}, one has $\dim_\text{hyp}(f)=\dim_H(\J_{\text{rad},f})$ for any $f\in\rat_d$, where $\J_{\text{rad},f}\subset\J_f$ is the set of radial points of $f$ (see \cite{McMullen2} page 541 for the definition). Moreover, Proposition 6.1 of \cite{Urbanski} implies that $\J_f\setminus\J_{\text{rad},f}$ is countable whenever $f$ is Misiurewicz.
\par Let now $f\in\rat_d$ be $k$-Misiurewicz. If $k=2d-2$ then $\J_f=\p^1$ and, by Theorem \ref{tm2intro}, the proof is over (see Proposition \ref{fatou}). Assume now that $k<2d-2$. According to Lemma \ref{parabolicimplosion}, for any $k$-Misiurewicz map $f\in\rat_d$ and any neighborhood $U\subset \rat_d$ of $f$ and any $\epsilon>0$, there exists $g\in\mathfrak{M}_k\cap U$ such that $\dim_\text{hyp}(g)=\dim_H(\J_g)\geq 2-\epsilon$. As $g\in U$, combined with Theorem \ref{tm2intro}, this gives the wanted estimate.\hfill$\Box$

\section{Linearization along infinite repelling orbits}

~

\par We aim here to prove the following linearization process along infinite repelling orbits:

~

\begin{prop}
Let $(f_\lambda)_{\lambda\in\B(0,r)}$ be a holomorphic family of degree $d$ rational maps. Let $E_0\subset\J_0$ be a compact $f_0$-hyperbolic set and $h:\B(0,r)\times E_0\longrightarrow\p^1$ be its dynamical holomorphic motion. Let $w_0\in E_0$ and, for $\lambda\in\mathbb{B}(0,r)$, $w(\lambda)\pe h_\lambda(w_0)$. Then there exists constants $\rho,C>0$ and for $n\geq1$, continuous functions $\rho_n:\B(0,r)\longrightarrow\R_+^*$ and holomorphic injections $\psi^{(n)}_{0,\lambda}:\D(0,\rho_n(\lambda))\rightarrow\D(0,2\rho_n(\lambda))$ and $\psi^{(n)}_{1,\lambda}:\D(0,\rho)\rightarrow\D(0,2\rho)$ holomorphically depending on $\lambda\in\mathbb{B}(0,r)$ and satisfying:
\begin{itemize}
\item $ f^n_\lambda(z+w(\lambda))-f_\lambda^n(w(\lambda))=\psi^{(n)}_{1,\lambda}\Big(\big(f^n_\lambda\big)'(w(\lambda)).\psi^{(n)}_{0,\lambda}(z)\Big)$,
\item $\rho_n(\lambda)\pe\frac{\rho}{2}|(f_\lambda^n)'(w(\lambda))|^{-1}$,
\item $|\psi^{(n)}_{i,\lambda}(z)-z|\leq C|z|^2$,
\end{itemize}
for all $\lambda\in\mathbb{B}(0,r)$ and all $z\in\D(0,\rho_n(\lambda))$, $z\in\D(0,\rho)$. Moreover, we may assume that $C\rho<1$.
\label{corlinearisation}
\end{prop}

~

\par The proof uses two ingredients. The first one is a linearization principle for chains of contractions. The second one is a construction of "good" inverse branches along an orbit in $E_0$ to which apply this principle.

\subsection{Linearization for chains of contractions}

\begin{defi}
Let $a\leq b<0$. A $(a,b)$-\emph{chain of holomorphic contractions} is a sequence of holomorphic maps

\begin{center}
$\xymatrix {\relax
g_j:\D(0,\eta) \ar[r] & \D(0,\eta)}$
\end{center}
for which there exists $a_j\leq b_j<0$ such that $e^a|z|\leq e^ {a_j}|z|\leq|g_j(z)|\leq e^ {b_j}|z|\leq e^b|z|$, for all $z\in\D(0,\eta)$. We say that the chain \emph{depends holomorphically} on $\lambda\in\B(0,r)$ if $g_j(z)=g_{j,\lambda}(z)$ is holomorphic on $(\lambda,z)\in\B(0,r)\times\D(0,\eta)$ for each $j\geq1$.
\end{defi}

~

\par The linearization principle for chains of contractions is a version of the classical Koenigs Theorem in the non-autonomous setting. The idea of the proof is essentially the same as for the classical Theorem of Koenigs. In a more general setting, Berteloot, Dupont and Molino establish a similar result (see \cite{BDM}).

~

\begin{tm}[Linearization]
Let $a\leq b<0$ and $\eta\leq1$. Let $(g_{j,\lambda})_{j\geq1}$ be a $(a,b)$-chain of holomorphic contractions on $\D(0,\eta)$ which holomorphically depend on $\lambda\in\B(0,r)$. Denote by $l_{j,\lambda}$ the linear part of $g_{j,\lambda}$ at $0$ and assume that $2b_j-a_j\leq-\delta$ for all $j\geq1$ and some $\delta>0$. Then there exists a constant $0<\rho\leq\eta/2$ and a sequence $\varphi_{j,\lambda}:\D(0,\rho)\longrightarrow\D(0,2\rho)$ of holomorphic injections holomorphically depending on $\lambda\in\mathbb{B}(0,r)$ such that the following diagram commutes:

\begin{center}
$\xymatrix {\relax
\D(0,\rho) \ar[r]^{g_{1,\lambda}} \ar[d]^{\varphi_{1,\lambda}} & \D(0,\rho) \ar[r]^{g_{2,\lambda}} \ar[d]^{\varphi_{2,\lambda}} & \cdots \ar[r]^{g_{j-1,\lambda}} & \D(0,\rho) \ar[r]^{g_{j,\lambda}} \ar[d]^{\varphi_{j,\lambda}} & \D(0,\rho) \ar[r]^{g_{j+1,\lambda}} \ar[d]^{\varphi_{j+1,\lambda}} & \cdots\\
\D(0,2\rho) \ar[r]^{l_{1,\lambda}} & \D(0,2\rho) \ar[r]^{l_{2,\lambda}} & \cdots \ar[r]^{l_{j-1,\lambda}} & \D(0,2\rho) \ar[r]^{l_{j,\lambda}} & \D(0,2\rho) \ar[r]^{l_{j+1,\lambda}}  & \cdots}$
\end{center}

Moreover, there exists $C_0>0$ such that
\begin{center}
$|\varphi_{j,\lambda}(z)-z|\leq\displaystyle\frac{2}{\eta}e^{b-a}C_0|z|^2$, $\forall\ j\geq1$, $\forall\ \lambda\in\mathbb{B}(0,r)$, $\forall\ z\in\D(0,\rho)$.
\end{center} \label{lmlinearisation2}
\end{tm} 

\begin{proof} For any $z\in\D(0,\eta)$ and $\lambda\in\mathbb{B}(0,r)$ we have
\begin{center}
$\displaystyle|g_{j,\lambda}(z)-l_{j,\lambda}(z)|=\Bigg|\sum_{n\geq2}\frac{g_{j,\lambda}^{(n)}(0)}{n!}z^n\Bigg|\leq|z|^2\sum_{n\geq2}\frac{|g_{j,\lambda}^{(n)}(0)|}{n!}.|z|^{n-2}$.
\end{center} 
By Cauchy's inequalities we get $\sum\limits_{n\geq2}\frac{|g_{j,\lambda}^{(n)}(0)|}{n!}.|z|^{n-2}\leq e^b\sum\limits_{n\geq2}\frac{|z|^{n-2}}{\eta^{n-1}}=\frac{e^{b_j}}{\eta}\sum\limits_{k\geq0}\Big(\frac{|z|}{\eta}\Big)^{k}$ and thus
\begin{eqnarray}
\displaystyle|g_{j,\lambda}(z)-l_{j,\lambda}(z)|\leq2\frac{e^{b_j}}{\eta}|z|^2\textup{ for }z\in\D(0,\frac{\eta}{2})\textup{, }\lambda\in\mathbb{B}(0,r)\textup{ and }j\geq1.
\label{lmlinearisation1}
\end{eqnarray} 

~

\par For $j\geq1$, $n\geq1$, $\lambda\in\mathbb{B}(0,r)$ and $z\in\D(0,\eta)$, we set:
\begin{eqnarray*}
\varphi_{j,n,\lambda}(z)\pe l_{j,\lambda}^{-1}\circ\cdots\circ l_{j+n-1,\lambda}^{-1}\circ g_{j+n-1,\lambda}\circ\cdots\circ g_{j,\lambda}(z)
\end{eqnarray*} 
and $\varphi_{j,0,\lambda}=\id$. By induction on $n\geq0$, we will establish the assertion
\begin{eqnarray*}
(\mathcal{A}_n) \ : \ \hspace{0.5cm} \forall j\geq1,\ \forall z\in\D(0,\frac{\eta}{2}), \ |\varphi_{j,n+1,\lambda}(z)-\varphi_{j,n,\lambda}(z)|\leq \frac{2}{\eta}e^{b-a-n\delta}|z|^2.
\end{eqnarray*}
By \ref{lmlinearisation1} we have $(\mathcal{A}_0)$ : $\displaystyle |\varphi_{j,1,\lambda}(z)-\varphi_{j,0,\lambda}(z)|\leq e^{-a}|g_{j,\lambda}(z)-l_{j,\lambda}(z)|\leq \frac{2}{\eta}e^{b-a}|z|^2$.
\par Suppose now that $(\mathcal{A}_n)$ is true. Since $|g_{j,\lambda}(z)|\leq e^{b_j}|z|\leq\eta/2$ when $z\in\D(0,\eta/2)$ we get
\begin{eqnarray*}
|\varphi_{j+1,n+1,\lambda}(g_j(z))-\varphi_{j+1,n,\lambda}(g_j(z))|\leq\frac{2}{\eta}e^{b-a-n\delta}|g_{j,\lambda}(z)|^2\leq\frac{2}{\eta}e^{b-a-n\delta+2b_j}|z|^2,
\end{eqnarray*} 
which, post-composing by $l_{j,\lambda}^{-1}$, yields $(\mathcal{A}_{n+1})$.

~

\par As $b<0,$ the sequence $(\varphi_{j,n,\lambda})_{n\geq1}$ uniformly converges to a holomorphic function $\varphi_{j,\lambda}=\id+\sum_{n\geq0}(\varphi_{j,n+1,\lambda}-\varphi_{j,n,\lambda})$ on $\D(0,\frac{\eta}{2})\times\B(0,r)$. Moreover, setting $C_0\pe (1-e^\delta)^{-1}$ and $\rho\pe\frac{1}{2}\min\big(\frac{\eta}{2},\frac{\eta}{2e^{b-a}C_0}\big)$, one has
\begin{eqnarray}
\displaystyle|\varphi_{j,\lambda}(z)-z|\leq\frac{2}{\eta}e^{b-a}C_0|z|^2\leq\frac{1}{2}|z|\textup{ for }z\in\D(0,2\rho)\textup{ and }\lambda\in\mathbb{B}(0,r).\label{equationlinearisation1}
\end{eqnarray} 
Making $n\rightarrow\infty$ in $\varphi_{j+1,n,\lambda}\circ g_{j,\lambda}=l_{j,\lambda}\circ \varphi_{j,n+1,\lambda}$ we get $\varphi_{j+1,\lambda}\circ g_{j,\lambda}=l_{j,\lambda}\circ \varphi_{j,\lambda}$.
\par Let us set $\psi_{j,\lambda}(z)\pe \varphi_{j,\lambda}(z)-z$. The first inequality in \ref{equationlinearisation1} gives $|\psi_{j,\lambda}'(z)|\leq\frac{1}{2}$ on $\D(0,\rho)$. This implies that $\varphi_{j,\lambda}$ is injective on $\D(0,\rho)$.\end{proof}

\subsection{Families of inverse branches}

\par Let us return to the case where $(f_\lambda)_{\lambda\in\B(0,r)}$ is a holomorphic family and $E_0\subset\J_0$ is a compact $f_0$-hyperbolic subset of $\p^1$. Let $h$ be the dynamical holomorphic motion of $E_0$. Using the compactness of $E_0$, we will construct inverse branches of $f_\lambda$, at any depth, along the orbit of $h_\lambda(w_0)$.

~

\begin{lm}
The assumptions and notations are those of Proposition \ref{corlinearisation}. Let $n\geq1$, $w_0\in E_0$ and $w(\lambda)\pe h_\lambda(w_0)$. There exists constants $B>K>1$, $\eta>0$, independent of $w_0$ and $n$, and a sequence $(f_{\lambda,j}^{-1})_{j\geq1}$ of inverse branches of $f_\lambda$ holomorphically depending on $\lambda\in\mathbb{B}(0,r)$ such that
\begin{enumerate}
\item $f_{\lambda,j}^{-1}$ is defined on $\D(z_j(\lambda),\eta)$ for every $j\geq1$ and $\lambda\in\mathbb{B}(0,r)$, where 
\begin{center}
$z_j(\lambda)\pe f_{\lambda,j-1}^{-1}\circ\cdots\circ f_{\lambda,1}^{-1}(f_\lambda^n(w(\lambda)))$,
\end{center}
\item $z_{i}(\lambda)=f_\lambda^{n-i+1}(w(\lambda))$ for every $1\leq i\leq n+1$,
\item for any $j\geq1$, $\lambda\in\mathbb{B}(0,r)$ and $z\in\D(z_j(\lambda),\eta)$ one has
\begin{eqnarray}
\frac{1}{B}|z-z_{j}(\lambda)|\leq|f_{\lambda,j}^{-1}(z)-z_{j+1}(\lambda)|\leq\frac{1}{K}|z-z_{j}(\lambda)|.\label{lipschitz}
\end{eqnarray}
\end{enumerate}
\label{lmbrancheinverse}
\end{lm}

\begin{proof} Let $K>1$ be the hyperbolicity constant of $E_0$ and $\mathcal{N}_\delta$ be a $\delta$-neighborhood of $E_0$ in $\p^1$. Up to reducing $r$, $\delta$ and $K$ we may assume that $|f_\lambda'(z)|\geq K$ for all $(\lambda,z)\in\B(0,r)\times\mathcal{N}_\delta$. Set $B\textup{:=}\max_{\lambda\in\mathbb{B}(0,r),z\in\mathcal{N}_\delta}|f_\lambda'(z)|$. Let us reformulate Lemma \ref{cormvtholo}. There exists $\alpha>0$ such that $f_{f_0^{q-1}(w_0),\lambda}^{-1}$ is defined on $\B(0,r)\times\D(f_0^q(w_0),\alpha)$ for any $q\geq1$ and satisfies
\begin{itemize}
	\item $f_\lambda^{q-1}(w(\lambda))=f_{f_0^{q-1}(w_0),\lambda}^{-1}\circ f_\lambda^q(w(\lambda))$ for every $\lambda\in\B(0,r)$.
	\item $\frac{1}{B}|z-w|\leq|f_{f_0^{q-1}(w_0),\lambda}^{-1}(z)-f_{f_0^{q-1}(w_0),\lambda}^{-1}(w)|\leq\frac{1}{K}|z-w|$ for $\lambda\in\B(0,r)$ and $z,w\in\D(f_0^q(w_0),\alpha)$.
\end{itemize}

~

\par Fix $\varepsilon>0$ such that $\frac{K\varepsilon}{K-1}<\frac{\alpha}{4}$. By compactness of $E_0$ we may find $n_0\geq0$ and $m\geq1$ such that $|f_0^{m+n_0}(w_0)-f_0^{n_0}(w_0)|\leq\frac{\varepsilon}{2}$. It is clearly sufficient to make the construction for $f_0^{n_0}(w_0)$. We thus will assume that $|f_0^{m}(w_0)-w_0|\leq\frac{\varepsilon}{2}$ and, after shrinking $r$, that:
\begin{center}
$|f_\lambda^m(w(\lambda))-w(\lambda)|\leq\varepsilon$ for every $\lambda\in\mathbb{B}(0,r)$.
\end{center} 
Set $z_i(\lambda)\pe f^{n-i+1}_\lambda(w(\lambda))$ and $f_{\lambda,i}^{-1}(z)\pe f_{f_0^{n-i}(w_0),\lambda}^{-1}(z)$ for $1\leq i\leq n$, $z_{n+1}(\lambda)\pe w(\lambda)$. Denote $\varepsilon_k\pe \varepsilon\big(1+\frac{1}{K^m}+\ldots+\frac{1}{K^{km}}\big)$, then $\varepsilon_k<\frac{\alpha}{4}$ by our choice of $\varepsilon$.

~

\par Set $\eta\pe\alpha/4$. Observe that the map $f_{\lambda,n}^{-1}$ is defined on $\D(f_\lambda(w(\lambda)),\eta)$ and takes values in $\D(w(\lambda),\eta/K)\subset\D(f^m_\lambda(w(\lambda)),\alpha/2)$. Since $z_{n+1}(\lambda)=w(\lambda)$ we obviously have $|z_{n+1}(\lambda)-f^m_\lambda(w(\lambda))|\leq \varepsilon_0=\varepsilon$. According to the above reformulation of Lemma \ref{cormvtholo}, we may take $f_{\lambda,n+1}^{-1}\pe f_{f_0^{m-1}(w_0),\lambda}^{-1}$ on $\D(z_{n+1}(\lambda),\eta)$. By induction one easily sees that the inverse branches $f_{\lambda,j}^{-1}$ exist, that the points $z_j(\lambda)$ are well-defined for $j\geq n+1$ and satisfy
\begin{eqnarray*}
|z_{j}(\lambda)-f_\lambda^{m-r}(w(\lambda))|\leq\frac{\varepsilon_k}{K^r}
\label{inegkj}
\end{eqnarray*}
when $j$ is of the form $j=n+km+r+1$ where $0\leq r\leq m-1$ and $k\geq0$. We may moreover observe that there exists $1\leq r(j)\leq m$ such that $z_j(\lambda)\in\D(f_\lambda^{r(j)}(w(\lambda)),\eta)$ for all $j\geq1$.\end{proof}

\subsection{Linearization along repelling orbits}
We are ready for proving Proposition \ref{corlinearisation}:

~

\par \textit{Proof of Proposition \ref{corlinearisation}.} Recall that $w_0\in E_0$ is fixed and that $w(\lambda)= h_\lambda(w_0)$. Let $K$, $B$, $f_{\lambda,j}^{-1}$ and $z_j(\lambda)$ be given by Lemma \ref{lmbrancheinverse}. In particular $z_1(\lambda)=f_\lambda^n(w(\lambda))$. For $j\geq1$, $\lambda\in\mathbb{B}(0,r)$ and $z\in\D(0,\eta)$ we set:
\begin{center}
$g_{j,\lambda}(z)\pe f_{\lambda,j}^{-1}(z+z_j(\lambda))-z_{j+1}(\lambda)$,
\end{center}
\begin{center}
$l_{j,\lambda}(z)\pe \big(f_{\lambda,j}^{-1}\big)'(z_j(\lambda)). z$.
\end{center}
Up to a translation $g_{j,\lambda}$ is an inverse branch of $f_\lambda$. Let us stress that one must consider the full sequence $(g_{j,\lambda})_{j\geq1}$ for applying Theorem \ref{lmlinearisation2}. We then set
\begin{center}
$a_j\pe\log\inf\limits_{\genfrac{}{}{0pt}{}{|z|\leq\eta}{\|\lambda\|\leq r}}|g_{j,\lambda}'(z)|$, \ $a\pe\inf_{j\geq1}a_j$, \ $b_j\pe\log\sup\limits_{\genfrac{}{}{0pt}{}{|z|\leq\eta}{\|\lambda\|\leq r}}|g_{j,\lambda}'(z)|$ \ and \ $b\pe\sup_{j\geq1}b_j$.
\end{center} 
According to Lemma \ref{lmbrancheinverse}, for every $\lambda\in\mathbb{B}(0,r)$, $j\geq1$ and $z\in\D(0,\eta)$ we have:
\begin{eqnarray}
\frac{1}{B}|z|\leq e^a|z|\leq e^{a_j}|z|\leq|g_{j,\lambda}(z)|\leq e^{b_j}|z|\leq e^b|z|\leq \frac{1}{K}|z|.
\label{lipgj1}
\end{eqnarray} 
Shrinking $\eta$, we may assume that $e^{2b_j-a_j}\leq e^{b/2}$ for any $j\geq1$ and taking the $\sup_{j\geq1}$ that $\sup_{j\geq1}e^{2b_j-a_j}\leq e^{b/2}$.

~

\par We will denote by $f_\lambda^{-n}$ the inverse branch of $f_\lambda^n$ satisfying $f_\lambda^{-n}(f_\lambda^n(w(\lambda)))=w(\lambda)$ obtained by composing the branches given by Lemma \ref{lmbrancheinverse}. We now apply Theorem \ref{lmlinearisation2} to the sequence $(g_{j,\lambda})_{j\geq1}$. This yields a chain linearization which we cut at $j=n$. We find in this way holomorphic injections $\tilde\psi^{(n)}_{1,\lambda}$ and $\psi^{(n)}_{0,\lambda}$ defined on $\D(0,\rho)$ and taking values in $\D(0,2\rho)$ holomorphically depending on $\lambda\in\B(0,r)$ and such that
\begin{eqnarray}
\psi^{(n)}_{0,\lambda}\Big(f_\lambda^{-n}\big(z+f_\lambda^n(w(\lambda))\big)-w(\lambda)\Big)=\big((f_\lambda^n)'(w(\lambda))\big)^{-1}\tilde\psi^{(n)}_{1,\lambda}(z)
\label{eqfonc}
\end{eqnarray}
for every $\lambda\in\B(0,r)$ and $z\in\D(0,\rho)$. Moreover, there exists $C>0$ such that
\begin{center}
$|\psi_{0,\lambda}^{(n)}(z)-z|\leq C|z|^2$ and $|\tilde \psi_{1,\lambda}^{(n)}(z)-z|\leq C|z|^2$
\end{center}
for any $\lambda\in\mathbb{B}(0,r)$ and $z\in\D(0,\rho)$. Up to reducing $\rho$ we may assume that $C\rho<1$ and $\psi^{(n)}_{1,\lambda}\pe\big(\tilde{\psi}^{(n)}_{1,\lambda}\big)^{-1}$ is defined and injective on $\D(0,\rho)$. Changing $C$ and reducing $\rho$ again, we also have $|\psi_{1,\lambda}^{(n)}(z)-z|\leq C|z|^2$. Moreover, for $z\in\D(0,\rho_n(\lambda))$, we have $|\psi^{(n)}_{0,\lambda}(z)|\leq2\rho_n(\lambda)$ and $|(f_\lambda^n)'(w(\lambda))\psi^{(n)}_{0,\lambda}(z)|\leq\rho$. To end the proof of the proposition, it remains to invert \ref{eqfonc}.

\section{A criterion for self-intersections of the bifurcation current}

~

\par We wish here to establish a criterion for a rational map to belong to the support of $T_\bif^k$. When $f_0$ has $q$ critical points eventually falling in a compact $f_0$-hyperbolic set, we show that $f_0\in\supp(T_\bif^q)$ as soon as the fibers of the activity map defined in section $3$ have codimension $q$. This leads to a sufficient condition for a $q$-Misiurewicz rational map to lie in the support of $T_\bif^q$, which is precisely the statement of Theorem \ref{tmtransverseintro}.

\subsection{No self-intersections for the bifurcation current of a critical point} The aim of this subsection is to prove that in any holomorphic family of rational maps, the bifurcation current of a critical point never has self-intersections. This has already been proved by Dujardin and Favre for polynomal families (see \cite{favredujardin} Proposition 6.9). As a consequence, we see that for a parameter to lie in the support of the $k$-th self-intersection $T_\bif^q$, at least $q$ distinct critical points need to be active at this parameter.

~

\begin{tm}
Let $(f_\lambda)_{\lambda\in X}$ be a holomorphic family of degree $d$ rational maps with $2d-2$ marked critical points. Assume that there exists a holomorphic family $(F_\lambda)_{\lambda\in X}$ of polynomial lifts of the family $(f_\lambda)_{\lambda\in X}$ and that the $2d-2$ marked critical points admit holomorphic lifts $\tilde c_j:X\longrightarrow\C^2\setminus\{0\}$ for which $\textup{det}D_zF_\lambda=\prod_{j=1}^{2d-2} \tilde c_j(\lambda)\wedge z$. Then
\begin{center}
$dd^cG_\lambda(\tilde c_i(\lambda))\wedge dd^cG_\lambda(\tilde c_i(\lambda))\equiv0$.
\end{center}
In particular, for any $1\leq q\leq \max(\dim_\C X,2d-2)$, we get
\begin{center}
$T_\bif^q=\displaystyle\sum_{i_1\neq i_2\neq\cdots\neq i_q}\bigwedge_{j=1}^qdd^cG_\lambda(\tilde c_{i_j}(\lambda))$.
\end{center}
\label{tmddcg}
\end{tm}

\begin{proof} The idea of the proof is the same as the one used by Bassanelli and Berteloot to prove Fact page 224 of \cite{BB1}. If $c_j$ is passive on $X$ or if $\dim X\leq 1$, there is nothing to do. We thus assume that $\dim X=m\geq2$ and $dd^cG_\lambda(\tilde c_j(\lambda))\neq0$ on $X$. Set $g(\lambda)\pe G_\lambda(\tilde c_j(\lambda))$ and consider $\lambda_0\in\supp(dd^cg)$. As the desired property is local, we may assume that $X=\C^m$. Let $\B(\lambda_0,r)$ be a ball of $\C^m$ centered at $\lambda_0$.
\par Let $V$ be a $2$-dimensional affine subspace of $\B(\lambda_0,r)$ and let $\mu_V\pe(dd^c(g|_V))^2$. By a slicing argument, $\mu_V$ is the slice of $(dd^cg)^2$ by $V$. To prove that $(dd^cg)^2=0$ on $\B(\lambda_0,r)$, it suffices to show that $\mu_V$ vanishes on $\frac{1}{2}\B$ for any ball $\B\Subset V$, for any $2$-dimensional affine subspace $V$ of $\B(\lambda_0,r)$.

~

\par Let $h$ be the solution to the Dirichlet-Monge-Amp\`ere problem on $\B$ with data $g|_V$ on the boundary. The function $h$ is continuous on $\overline{\B}$, concides with $g|_V$ on $\partial\B$ and is maximal $p.s.h$ on $\B$ (see \cite{BedfordTaylor}). By maximality of $h$ one has $g|_V\leq h$ on $\overline{\B}$. For $\varepsilon>0$, we set
\begin{center}
$\mathcal{S}_\varepsilon\pe\{\lambda\in\frac{1}{2}\B$ / $0\leq h(\lambda)-g|_V(\lambda)\leq\varepsilon\}$.
\end{center} 
By a Theorem of Briend-Duval (see \cite{Sibony} Theorem A.10.2), there exists a constant $C$ only depending on $g|_V$ and $\B$ such that $\mu_V(\mathcal{S}_\varepsilon)\leq C\varepsilon$. To show that $\mu_V$ vanishes on $\frac{1}{2}\B$, it thus suffices to prove that $\supp(\mu_V)\cap\frac{1}{2}\B\subset\mathcal{S}_\varepsilon$ for any $\varepsilon>0$.

~

\par Set now $\Per(n)\pe\{\lambda\in \B(\lambda_0,r)$ / $f_\lambda^n(c_j(\lambda))=c_j(\lambda)\}$ for $n\geq1$. As the activity locus of $c_j$ intersects $\B(\lambda_0,r)$, the analytic sets $\Per(n)$ are curves of $\B(\lambda_0,r)$ and in $\Per(n)$ one has $\partial\big(\Per(n)\cap\B\big)=\Per(n)\cap\partial\B$.
\par Since for any $n\geq1$ the function $g|_V$ is harmonic on the curve $\Per(n)\cap\B$, the function $h-g|_V$ is subharmonic on $\Per(n)\cap\B$. As $h-g|_V$ vanishes at every  point of $\Per(n)\cap\partial\B$, the maximum principle yields $h\leq g|_V$ on $\Per(n)\cap\B$ and the maximality of $h$ on $\B$ gives $h-g|_V\equiv0$ sur $\Per(n)\cap\B$.
\par By Montel's Theorem, we have $\supp(dd^c(g|_V))\subset\overline{\bigcup_{n\geq1}\Per(n)}$. The function $h-g|_V$ being continuous, this implies $h-g|_V\equiv 0$ on $\supp(dd^c(g|_V))\cap\frac{1}{2}\B$. Finally, by definition of $\mu_V$, we get $\supp(\mu_V)\subset\supp(dd^c(g|_V))$, which means that $h-g|_V\equiv 0$ on $\supp(\mu_V)\cap\frac{1}{2}\B$. Thus we have shown that $\supp(\mu_V)\cap\frac{1}{2}\B\subset\mathcal{S}_\varepsilon$ for any $\varepsilon>0$.\end{proof}

\subsection{Statement and preliminaries}

\par Here we establish our criterion. We will use the notations introduced at the beginning of the subsection $3.1$. Let $(f_\lambda)_{\lambda\in\B(0,r)}$ be a holomorphic family of degree $d$ rational maps with $2d-2$ marked critical points parametrzed by a ball of $\C^m$. For $\lambda\in \B(0,r)$ we denote by $F_\lambda$ a non-degenerate homogeneous polynomial lift of $f_\lambda$ to $\C^2$. Up to reducing $r$, we may assume that the marked critical points $c_i(\lambda)$ admit holomorphic lifts $\tilde c_i(\lambda)$ to $\C^2$ and that $F_\lambda$ holomorphically depends on $\lambda$. Our precise result may be stated as follows:

~

\begin{tm}
Let $(f_\lambda)_{\lambda\in \B(0,r)}$ be a holomorphic family of degree $d$ rational maps with $2d-2$ marked critical points parametrized by a ball of $\C^m$. Assume that there is $k_0\geq1$ and $1\leq q\leq m$ such that $\overline{\{f_0^n(c_j(0))/ n\geq k_0\text{ and }1\leq j\leq q\}}$ is a compact $f_0$-hyperbolic set. Let $\chi$ be the activity map of $(f_\lambda)_{\lambda\in\B(0,r)}$ at $\lambda=0$. If $\chi^{-1}\{0\}$ has pure codimension $q$, then $f_0\in\textup{supp}(dd^cG_\lambda(\tilde c_1(\lambda))\wedge\cdots\wedge dd^cG_\lambda(\tilde c_q(\lambda)))$.\label{tmcourantsbif}
\end{tm}

\par Our method for proving Theorem \ref{tmcourantsbif} is directly inspired by the work of Buff and Epstein \cite{buffepstein}. We give here a more general criterion. Let us remark that if $f_0$ is geometrically finite and $\chi$ is invertible, the proof of Theorem \ref{tmcourantsbif} is exactly the same as the one of Buff and Epstein. Lemma \ref{lmapproxmis} guarantees that any $q$-Misiurewicz map can be approximated by geometrically finite $q$-Misiurewicz maps and therefore, it suffices to prove Theorem \ref{tmcourantsbif} in the geometrically finie case. Making the proof for any $q$-Misiurewicz map allows us to establish estimates for the pointwise dimension of the bifurcation measure at Misiurewicz maps (see section \ref{sectionmubifloc}).

~

\par \textit{Proof of Theorem \ref{tm2intro}.} Let $f$ be a $k$-Misiurewicz map which is not a flexible Latt\`es map. Up to taking a finite branch cover of $\rat_d$, we can assume that the family has $2d-2$ marked critical points in a neighborhood of $f$. Assume that $c_1,\ldots,c_k\in\J_f$. Since $c_{k+1},\ldots,c_{2d-2}$ are passive at $f$, Theorem \ref{tmddcg}  implies that $f\notin\supp(T_\bif^{k+1})$ and gives
\begin{eqnarray*}
T_{\textup{bif}}^k=k!\bigwedge_{i=1}^kdd^cG_\lambda(\tilde c_i(\lambda))
\end{eqnarray*}
in a neighborhood of $f$. B Theorem \ref{tmcourantsbif} and Lemma \ref{lmvariete}, we have $f\in\supp(T_\bif^k)$.\hfill$\Box$

~
 
\par After replacing $\C^m$ by a transversal subspace $\C^q$ to $\chi^{-1}\{0\}$, the situation may be reduced to the case where $m=q$ and $\chi^{-1}\{0\}$ is discrete:

~

\begin{lm}
Let $1\leq q\leq m$ and let $u_1,\ldots,u_q$ be continuous psh functions on a ball $\B(0,r)\subset\C^m$. If $0\in\supp\big(dd^c(u_1|_{\C^q})\wedge\cdots\wedge dd^c(u_q|_{\C^q})\big)$, then $0\in\supp\big(dd^cu_1\wedge\cdots\wedge dd^cu_q\big)$.
\end{lm}

\begin{proof} By definition of the slice $\big(dd^cu_1\wedge\cdots\wedge dd^cu_q\big)|_{\C^q\cap\B(0,r)}$, we have
\begin{center}
$\big(dd^cu_1\wedge\cdots\wedge dd^cu_q\big)|_{\C^q\cap\B(0,r)}=dd^c(u_1|_{\C^q\cap\B(0,r)})\wedge\cdots\wedge dd^c(u_q|_{\C^q\cap\B(0,r)})$.
\end{center}
Moreover, if $0\in\supp(dd^c(u_1|_{\C^q\cap\B(0,r)})\wedge\cdots\wedge dd^c(u_q|_{\C^q\cap\B(0,r)}))$, then the positive measure $dd^c(u_1|_{\C^q\cap\B(0,r)})\wedge\cdots\wedge dd^c(u_q|_{\C^q\cap\B(0,r)})$ is a non-zero on any neighborhood $V$ of $0$ in $\C^q\cap\B(0,r)$. By continuity of $u_1,\ldots u_q$, the psh functions $u_j(\cdot+a)|_{\C^q\cap\B(0,r)}$ converge uniformly locally to the psh function $u_j|_{\C^q\cap\B(0,r)}$. Therefore, the measures
\begin{center}
$\mu_a\pe dd^c(u_1(\cdot+a)|_{\C^q\cap\B(0,r)})\wedge\cdots\wedge dd^c(u_q(\cdot+a)|_{\C^q\cap\B(0,r)})$
\end{center}
converge to $\mu_0=dd^c(u_1|_{\C^q\cap\B(0,r)})\wedge\cdots\wedge dd^c(u_q|_{\C^q\cap\B(0,r)})$ when $a\rightarrow0$. In particular, $\mu_a$ is also non-zero on $\{\lambda\in V$ / $\lambda+a\in\B(0,r)\}$ for any $a\in\C^{m-q}\cap\B(0,r)$ close enough to $0$. Finally, if $dd^cu_1\wedge\cdots\wedge dd^cu_q$ were zero in a neighborhood of $0$, then almost all its slices would be zero.\end{proof}

~

\par Note then $\B_r$ for $\B(0,r)$ and $\chi$ for the activity map:
\begin{eqnarray*}
\chi:\B_r & \longrightarrow & \C^q\\
\lambda & \longmapsto & (\chi_1(\lambda),\ldots,\chi_q(\lambda)).
\end{eqnarray*}
After reducing $r$ we may assume that $\chi^{-1}\{0\}=\{0\}$.

~

\begin{defi}
For $\epsilon>0$ denote by $\D_\epsilon^q\pe\D(0,\epsilon)\times\cdots\times\D(0,\epsilon)$ a polydisc of $\C^q$ centered at $0$. Recall that $m_{n,i}(0)=(f_0^n)'(\nu_{0,i}(0))=(f_0^n)'(f_0^{k_0}(c_i(0)))$ and $|m_{n+1,i}(0)|\geq K|m_{n,i}(0)|$ where $K>1$. We define a sequence of dilations $D_n$ by: $D_0\pe\textup{id}_{\C^q}$ and
\begin{eqnarray*}
D_n:\C^q & \longrightarrow & \C^q\\
(x_1,\ldots,x_q) & \longmapsto & (m_{n,1}(0). x_1,\ldots,m_{n,q}(0). x_q).
\end{eqnarray*}
\end{defi}

~

\par For every $n\geq0$ we set $E_n\pe D_n^{-1}(\D_\epsilon^q)$ and denote by $\Omega_n$ the connected component of $\chi^{-1}(E_n)$ containing $0$. The following properties are either obvious or classical. The third assertion is a Lojaciewicz type inequality:

~

\begin{prop}
Up to reducing $r$ and $\epsilon$:
\begin{enumerate}
\item $\Omega_{n+1}\subset\Omega_n$ for $n\geq0$ and $(\Omega_n)_{n\geq0}$ is a neighborhood basis of $0$ in $\B_r$.
\item $\chi:\Omega_n\longrightarrow E_n$ is a finite ramified cover of constant degree $p$ for any $n\geq0$.
\item There exists $C_1>0$ such that $\|\lambda\|\leq C_1\|\chi(\lambda)\|^{1/p}$ for $\lambda\in\Omega_0$.
\item There exists $C_2>0$ and $K>1$ such that $\Omega_n\subset\B(0,C_2\big(\frac{\epsilon}{K^n}\big)^{1/p})$ for any $n\geq1$.
\end{enumerate} 
\label{propchi}
\end{prop} 

~

The following Lemma refines a Lemma of Aspenberg (see \cite{Aspenberg1} Lemma 3.2):

~

\begin{lm}[Distortion]
Let $(f_\lambda)_{\lambda\in\B(0,r)}$ be a holomorphic family of degree $d$ rational maps. Let $E_0\subset\J_0$ be a compact $f_0$-hyperbolic set and $h:\B(0,r)\times E_0\longrightarrow\p^1$ be its dynamical holomorphic motion. Then there exists a constant $C>0$ such that for all $n\geq 1$ and all $w_0\in E_0$
\begin{center}
$\displaystyle\left|\frac{(f_\lambda^n)'(h_\lambda(w_0))}{(f_0^n)'(w_0)}-1\right|\leq e^{nC\|\lambda\|}-1$ for all $\lambda\in\B(0,r)$.
\end{center}
\label{lmdistorsion}
\end{lm}

\begin{proof} Up to reducing $r$, one may assume that $r\leq 1$. As $h$ is uniformly continuous on $\B(0,r)\times E_0$ and $|f_0'(z)|\geq K>1$ for any $z\in E_0$, the family $\{f_\lambda'\circ h_\lambda(z)/f_0'(z)-1\}_{z\in E_0}$ is equicontinuous on $\B(0,r)$. Since $f_\lambda'\circ h_\lambda(z)/f_0'(z)-1$ vanishes at $0$ for any $z\in E_0$, the Cauchy inequalities yield the existence of $C'>0$ independent of $z$ such that
\begin{eqnarray}
\left|\frac{f_\lambda'\circ h_\lambda(z)}{f_0'(z)}-1\right|\leq \frac{C'}{r}\|\lambda\|\text{ for any }\lambda\in\B(0,r).
\label{equicontinuityh}
\end{eqnarray}
Let $n\geq1$ and  $w_0\in E_0$. Set $w_j(\lambda)\pe f_\lambda^j(h_\lambda(w_0))=h_\lambda(f_0^j(w_0))$, $0\leq j\leq n-1$. Then
\begin{eqnarray*}
\frac{(f^n_\lambda)'(w_0(\lambda))}{(f_0^n)'(w_0)}-1=\prod_{j=0}^{n-1}\frac{f_\lambda'(w_j(\lambda))}{f_0'(w_j(0))}-1=\prod_{j=0}^{n-1}\left(\frac{f_\lambda'(w_j(\lambda))}{f_0'(w_j(0))}-1+1\right)-1
\end{eqnarray*}
and Lemma 15.3 of \cite{rudin2} gives:
\begin{eqnarray*}
\left|\frac{(f^n_\lambda)'(w_0(\lambda))}{(f_0^n)'(w_0)}-1\right|\leq\exp\left(\sum_{j=0}^{n-1}\left|\frac{f_\lambda'(w_j(\lambda))}{f_0'(w_j(0))}-1\right|\right)-1.
\end{eqnarray*}
The conclusion then follows from setting $C\pe C'/r$ in estimate \ref{equicontinuityh}.\end{proof}

~

\par The linearization process (see Proposition \ref{corlinearisation}) states that there exists continuous functions $\rho_{n,i}:\B_r\longrightarrow\R_+^*$ and local biholomorphisms $\psi_{0,\lambda}^{(n,i)},\psi_{1,\lambda}^{(n,i)}$ holomorphically depending on $\lambda\in \B_r$, respectively defined on $\D(0,\rho_{n,i}(\lambda))$, $\D(0,\rho)$ and such that
\begin{eqnarray}
& & f_\lambda^n(z+\nu_{0,i}(\lambda))=\psi^{(n,i)}_{1,\lambda}\Big(m_{n,i}(\lambda).\psi^{(n,i)}_{0,\lambda}\big(z\big)\Big)+\nu_{n,i}(\lambda)\textup{ for every }z\in\D(0,\rho_{n,i}(\lambda)).
\label{eqlinchi}
\end{eqnarray}

~

\par Here we have used the notations introduced at the beginning of Subsection $3.1$. Let us recall that $\rho_{n,i}(\lambda)=\frac{\rho}{2}|m_{n,i}(\lambda)|^{-1}$. The following Lemma clarifies how we are going to use \ref{eqlinchi} and the Distortion Lemma to renormalize the unstable critical orbits (see subsection 5.3):

~

\begin{lm}
\begin{enumerate}
\item There exists $n_0\geq1$ such that the following holds for $1\leq i\leq q$, $n\geq n_0$ and $\lambda\in\Omega_n$:
\begin{enumerate}
	\item $\rho_{n,i}(\lambda)\geq\frac{\epsilon}{|m_{n,i}(0)|}$,
	\item $\big|m_{n,i}(\lambda)\psi_{0,\lambda}^{(n,i)}\big(\chi_i(\lambda)\big)\big|\leq\rho$ and
	\item $f_\lambda^{n+k_0}(c_i(\lambda))=\xi_{n,i}(\lambda)=\psi^{(n,i)}_{1,\lambda}\Big(m_{n,i}(\lambda).\psi^{(n,i)}_{0,\lambda}\big(\chi_i(\lambda)\big)\Big)+\nu_{n,i}(\lambda)$.
\end{enumerate}
\item Modulo extraction, the sequence $(\nu_{n,i}(0))_{n\geq0}$ converges to $z_i\in\J_0$ for all $1\leq i\leq q$. Moreover, there exists sections $\sigma_i:\D(z_i,\alpha)\longrightarrow\C^2\setminus\{0\}$ of $\pi$, such that the maps
\begin{center}
 $\tilde\xi_{n,i}(\lambda)\pe\sigma_i\circ\xi_{n,i}(\lambda)$
\end{center}
are well-defined on $\Omega_n$ for every $n\geq n_0$ and such that $\sigma_i(\D(z_i,\alpha))\Subset\C^2\setminus\{0\}$.
\end{enumerate} 
\label{lmlinbif}
\end{lm}

\begin{proof} Up to reducing $\rho$ we may assume that for any $z\in\p^1$ there exists a section $\sigma_z:\D(z,2\rho)\longrightarrow \C^ 2\setminus\{0\}$ of $\pi$. By the Distortion Lemma \ref{lmdistorsion} and item 4 of Proposition \ref{propchi} there exists $C>0$, $K>1$ and $p\geq1$ such that
\begin{center}
$\displaystyle\left|\frac{m_{n,i}(\lambda)}{m_{n,i}(0)}\right|-1\leq\left|\frac{m_{n,i}(\lambda)}{m_{n,i}(0)}-1\right|\leq \exp\left(\frac{nC}{K^{n/p}}\right)-1$,
\end{center}
for every $\lambda\in\Omega_n$. Thus $\rho_{n,i}(\lambda)= \frac{\rho}{2|m_{n,i}(\lambda)|}\geq \frac{\rho}{3|m_{n,i}(0)|}$ for any $n$ large enough and every $\lambda\in\Omega_n$. Up to reducing $\epsilon$ we have $\epsilon\leq\rho/3$ and $(a)$ follows.
 \par Now if $\lambda\in\Omega_n$ we have $|\chi_i(\lambda)|\leq\frac{\epsilon}{|m_{n,i}(0)|}\leq\rho_{n,i}(\lambda)$ and Proposition \ref{corlinearisation} gives a constant $C>0$ such that $C\rho_{n,i}(\lambda)\leq C\rho\leq1$ and
\begin{center}
$\big|\psi_{0,\lambda}^{(n,i)}(\chi_i(\lambda))-\chi_i(\lambda)\big|\leq C|\chi_i(\lambda)|^2$.
\end{center}
Then $\big|\psi_{0,\lambda}^{(n,i)}(\chi_i(\lambda))\big|\leq2\rho_{n,i}(\lambda)=\rho|m_{n,i}(\lambda)|^{-1}$ and $(b)$ is proved. Since 
\begin{center}
$f_\lambda^{n+k_0}(c_i(\lambda))=f_\lambda^n(\xi_{0,i}(\lambda))=f_\lambda^n(\chi_i(\lambda)+\nu_{0,i}(\lambda))$,
\end{center}
the assertion \ref{eqlinchi} becomes 
\begin{center}
$\xi_{n,i}(\lambda)=f_\lambda^{n+k_0}(c_i(\lambda))=\psi_{1,\lambda}^{(n,i)}\Big(m_{n,i}(\lambda).\psi_{0,\lambda}^{(n,i)}\big(\chi_i(\lambda)\big)\Big)+\nu_{n,i}(\lambda)$,
\end{center}
which is $(c)$. Since $(\Omega_n)_{n\geq0}$ is a neighborhood basis of $0$ and since $\{\nu_{n,i}(0)$ / $n\geq0\}$ is relatively compact in $\p^1$, item $(2)$ is now clear.\end{proof}

~

\par The map $\chi:\Omega_0\longrightarrow \D_\epsilon^q$ being a finite branched cover, there exists a pure codimension $1$ analytic subset $\mathcal{R}$ of $\D_\epsilon^q$ such that $\chi:\Omega_0\setminus\chi^{-1}(\mathcal{R})\longrightarrow \D_\epsilon^q\setminus\mathcal{R}$ is a finite cover. Set $A_n\pe D_n\big(E_n\cap\mathcal{R}\big)$ for $n\geq1$. We have the following:

~

\begin{lm}
After taking a subsequence, the sequence $(A_n)_{n\geq1}$ of pure codimension $1$ analytic sets converges to a pure codimension $1$ analytic set $A_\infty$.
\end{lm}

\begin{proof} Write $\mathcal{R}=\{\lambda\in\D^q_\epsilon$ / $F(\lambda)=0\}$ and choose $\alpha_{(1)}$ a $q$-tuple such that the coefficient of $\lambda^{\alpha_{(1)}}$ in the power expansion of $F$ is non-zero. Then, we may write
\begin{eqnarray*}
F(\lambda)=a_{\alpha_{(1)}}\lambda^{\alpha_{(1)}}+\cdots+a_{\alpha_{(N)}}\lambda^{\alpha_{(N)}}+\sum_{\alpha_j>\alpha_{(1),j}}a_\alpha\lambda^\alpha.
\end{eqnarray*}
For a $q$-tuple $\alpha$, denote by $m_n^\alpha\pe m_{n,1}(0)^{\alpha_1}\ldots m_{n,q}(0)^{\alpha_q}$. Let $1\leq s\leq N$ be such that $\big|m_n^{\alpha_{(s)}}\big|=\inf\limits_{1\leq j\leq N}\big|m_n^{\alpha_{(j)}}\big|$. We then clearly have $A_n=\{\lambda\in\D^q_\epsilon$ / $m_n^{\alpha_{(s)}}. F\circ D_n^{-1}(\lambda)=0\}$ for $n\geq1$. The series expansion of $m_n^{\alpha_{(s)}}F\circ D_n^{-1}$ may be written
\begin{eqnarray*}
m_n^{\alpha_{(s)}}F\circ D_n^{-1}(\lambda)=\sum_{j=1}^Na_{\alpha_{(j)}}\frac{m_n^{\alpha_{(s)}}}{m_n^{\alpha_{(j)}}}\lambda^{\alpha_{(j)}}+\sum_{\alpha_j>\alpha_{(1),j}}a_\alpha\frac{m_n^{\alpha_{(s)}}}{m_n^\alpha}\lambda^\alpha.
\end{eqnarray*}
By the choice of $s$, the first sum defines a bounded sequence of polynomial functions and locally uniformly converges (up to extraction) to a polynomial $F_\infty$ on $\D^q_\epsilon$. Again by the choice of $s$, the second sum clearly uniformly converges to $0$ on $\D^q_\epsilon$. Thus the sequence $(A_n)_{n\geq1}$ converges (up to extraction) to a pure codimension $1$ analytic subset $A_\infty=\{\lambda\in\D^q_\epsilon$ / $F_\infty(\lambda)=0\}$ of $\D^q_\epsilon$.\end{proof}

~

\par We now set 
\begin{eqnarray*}
X_\infty\displaystyle\pe\overline{\bigcup_{n\geq 0}A_{n}}=A_\infty\cup\bigcup_{n\geq 0}A_{n}\text{ and }\dot\D_\epsilon^q\pe\D_\epsilon^q\setminus X_\infty.
\end{eqnarray*}
By a covering Theorem of Besicovitch (see \cite{Mattila} page 30), there exists an integer $P(q)\geq 1$ only depending on $q$ and a countable family $(\B_i)_{i\geq1}$ of closed balls such that $2\B_i\subset\dot\D_\epsilon^q$ and:
\begin{eqnarray}
\mathbf{1}_{\dot\D^q_\epsilon}\leq\sum_{i=1}^{+\infty}\mathbf{1}_{\B_i}\leq P(q).\mathbf{1}_{\dot\D^q_\epsilon}.
\label{besicovitch}
\end{eqnarray}
By Lebesgue convergence Theorem, this gives
\begin{center}
$\displaystyle\int_{\D^q_\epsilon}\sum_{i=1}^{+\infty}\mathbf{1}_{\B_i}.\mu=\sum_{i=1}^{+\infty}\mu(\B_i)$
\end{center}
for any finite Radon measure $\mu$ on $\D^q_\epsilon$. We may now start to prove Theorem \ref{tmcourantsbif}.

\subsection{First step: local bounds for $T_\bif^{q}$}
Recall that by Proposition \ref{propchi}, the map
\begin{eqnarray*}
\xymatrix {\relax \Omega_n \ar[rr]^{D_n\circ\chi} & &  \D^q_\epsilon}
\end{eqnarray*}
is a finite branched cover of degree $p$ on $\Omega_n$.

~

\begin{lm}
\begin{enumerate}
\item Let us set $\mu^{(q)}\pe dd^cG_\lambda(\tilde c_1(\lambda))|_{\B_r}\wedge\cdots\wedge dd^cG_\lambda(\tilde c_q(\lambda))|_{\B_r}$. Then for any $n\geq0$,
\begin{eqnarray*}
\frac{1}{P(q)}\sum_{i=1}^{+\infty}(D_n\circ\chi)_*\mu^{(q)}(\B_i) \leq \mu^{(q)}(\Omega_n) \leq \sum_{i=1}^{+\infty}(D_n\circ\chi)_*\mu^{(q)}(\B_i).
\end{eqnarray*}
\item There exists $n_0\geq1$ such that for any $n\geq n_0$, the exists $p$ inverse branches $S_{n,1,j},\ldots,S_{n,p,j}:\B_j\longrightarrow\Omega_n$ of $D_n\circ\chi$ so that the union $\bigcup_{1\leq l\leq p}S_{n,l,j}(\B_j)$ is a disjoint union and:
\begin{eqnarray*}
(D_n\circ\chi)_*\mu^{(q)}(\B_j)=d^{-q(n+k_0)}\sum_{l=1}^p\int_{\B_j}\bigwedge_{i=1}^qdd^cG_{S_{n,l,j}(x)}\big(\tilde\xi_{n,i}(S_{n,l,j}(x))\big).
\end{eqnarray*}
\end{enumerate}
\label{lmTk}
\end{lm} 

\begin{proof} $(1).$ As $G_\lambda(\tilde c_i(\lambda))$ is continuous for any $1\leq i\leq q$, the measure $\mu^{(q)}$ doesn't give mass to pluripolar sets. In particular, $\mu^{(q)}(\Omega_n)=\mu^{(q)}(\Omega_n\setminus (D_n\circ\chi)^{-1}(X_\infty))$. By \ref{besicovitch}, we have
\begin{eqnarray*}
\mathbf{1}_{\dot\D^q_\epsilon}\circ(D_n\circ\chi)\leq\sum_{i=1}^{+\infty}\mathbf{1}_{\B_i}\circ(D_n\circ\chi)\leq P(q).\mathbf{1}_{\dot\D^q_\epsilon}\circ(D_n\circ\chi),
\end{eqnarray*}
which gives
\begin{eqnarray*}
\mathbf{1}_{\Omega_n\setminus (D_n\circ\chi)^{-1}(X_\infty)}\leq\sum_{i=1}^{+\infty}\mathbf{1}_{\B_i}\circ(D_n\circ\chi)\leq P(q).\mathbf{1}_{\Omega_n\setminus (D_n\circ\chi)^{-1}(X_\infty)}.
\end{eqnarray*}
Evaluating the positive measure $\mu^{(q)}$, we find
\begin{eqnarray*}
\mu^{(q)}(\Omega_n)=\mu^{(q)}(\Omega_n\setminus (D_n\circ\chi)^{-1}(X_\infty))\leq\sum_{i=1}^{+\infty}\mu^{(q)}(\mathbf{1}_{\B_i}\circ(D_n\circ\chi))=\sum_{i=1}^{+\infty}(D_n\circ\chi)_*\mu^{(q)}(\B_i).
\end{eqnarray*}
We prove the same way the other inequality.

~ 

\par $(2)$. After item $(2)$ of Lemma \ref{lmlinbif}, the current $dd^cG_\lambda(\tilde\xi_{n,i}(\lambda))$ is well-defined on $\Omega_n$ for $1\leq i\leq q$ and $n\geq n_0$. For $\lambda\in\Omega_n$ we have $\pi\circ F_\lambda^{n+k_0}(\tilde c_i(\lambda))=\pi\circ\tilde\xi_{n,i}(\lambda)$. Owing to \ref{homogeneityGreen} and  \ref{functorialGreen} it yields
\begin{eqnarray}
d^{n+k_0}dd^cG_\lambda\big(\tilde c_i(\lambda)\big)=dd^cG_\lambda\big(\tilde\xi_{n,i}(\lambda)\big).
\label{formuleTbif2}
\end{eqnarray}
From \ref{formuleTbif2}, we deduce that
\begin{eqnarray*}
\mu^{(q)}(U_n)=\int_{U_n}\bigwedge_{i=1}^qdd^cG_\lambda\big(\tilde c_i(\lambda)\big)=d^{-q(n+k_0)}\int_{U_n}\bigwedge_{i=1}^qdd^cG_\lambda\big(\tilde\xi_{n,i}(\lambda)\big)
\end{eqnarray*}
for every Borel set $U_n\subset\Omega_n$ and every $n\geq n_0$.

~

\par By assumption $\B_j\cap X_\infty=\emptyset$. Since $D_n\circ\chi$ is a degree $p$ covering map on $\Omega_n\setminus\chi^{-1}(A_n)$, there exists $p$ inverse branches $S_{n,1,j},\ldots,S_{n,p,j}$ of $D_n\circ\chi$ defined on $\B_j$ and taking values in $\Omega_n$. It thus comes
\begin{eqnarray*}
\frac{1}{q!}d^{q(n+k_0)}(D_n\circ\chi)_*\mu^{(q)}(\B_j) & = & \sum_{l=1}^p\int_{S_{n,l,j}(\B_j)}(D_n\circ\chi)^*\bigwedge_{i=1}^qdd^cG_{S_{n,l,j}(x)}\big(\tilde\xi_{n,i}(S_{n,l,j}(x))\big),\\
 & = & \sum_{l=1}^p\int_{\B_j}\bigwedge_{i=1}^qdd^cG_{S_{n,l,j}(x)}\big(\tilde\xi_{n,i}(S_{n,l,j}(x))\big).
\end{eqnarray*}
This is the announced estimate.\end{proof}

\subsection{Second step: renormalization}

Recall that $\dot\D_\epsilon^q=\D_\epsilon^q\setminus X_\infty$. Let $\B\subset\D_\epsilon^q\setminus X_\infty$ be a closed ball and let $S_{n,1}\ldots,S_{n,p}$ be the inverse branches of $D_n\circ \chi$ defined on $\B$ and taking values in $\Omega_n$. The family $(S_{n,l})_{n\geq0}$ is a sequence of holomorphic maps defined on $\B$.

~

\begin{lm}
Let $\B$ be a closed ball contained in $\dot\D_\epsilon^q$. Modulo extraction the following sequences $(1\leq l\leq p$ and $1\leq i\leq q$) uniformly converge on $\B$ when $n\rightarrow+\infty$:
\begin{center}
$\begin{array}{l}
\xymatrix {\relax
S_{n,l} \ar[r] & 0,}\\
\xymatrix {\relax
\xi_{n,i}\circ S_{n,l}(x) \ar[r] & p_i(x_i),}\\
\xymatrix {\relax
G_{S_{n,l}(x)}\Big(\tilde \xi_{n,i}\circ S_{n,l}(x)\Big) \ar[r] & G_0\circ \sigma_i\big(p_i(x_i)\big),}
\end{array}$
\end{center}
where $p_i:\D(0,\epsilon)\longrightarrow\p^1$ is holomorphic and satisfies $p_i'(0)=1$ and $p_i(0)=z_i\in\J_0$. Moreover, there exists a constant $M>0$ such that:
\begin{center}
$\displaystyle\left\|G_{S_{n,l}}\left(\tilde\xi_{n,i}\circ S_{n,l}\right)\right\|_{L^{\infty}(\B)}\leq M$,
\end{center}
for any ball $\B\subset\dot\D_\epsilon^q$, any $n\geq n_0$, any $1\leq i\leq q$ and any $1\leq l\leq p$.
\label{lmconv}
\end{lm}

\begin{proof} By construction $S_{n,l}(\B)=\phi_{n,l}\circ D_n^{-1}(\B)\subset\Omega_n$. Thus the sequence $(S_{n,l})_{n\geq0}$ uniformly converges to $0$ on $\B$. Let us use Lemma \ref{lmlinbif}. Item $1.(c)$ allows us to write
\begin{center}
$\displaystyle\xi_{n,i}\big(S_{n,l}(x)\big)=\psi^{(n,i)}_{1,S_{n,l}(x)}\Big(m_{n,i}(S_{n,l}(x)).\psi^{(n,i)}_{0,S_{n,l}(x)}\Big(\frac{x_i}{m_{n,i}(0)}\Big)\Big)+\nu_{n,i}(S_{n,l}(x))$.
\end{center}
As the sequence $(\nu_{n,i}(0))_{n\geq0}$ converges up to extraction to $z_i\in\J_0$ and as $(\nu_{n,i})_{n\geq0}$ is equicontinuous, we see that $(\nu_{n,i}\circ S_{n,l})_{n\geq0}$ uniformly converges to $z_i$ on $\B$.

~

\par Let us now look at $\xi_{n,i}\circ S_{n,l}-\nu_{n,i}\circ S_{n,l}$. Set
\begin{center}
$u_{n,i}(x_i)\displaystyle\pe\psi^{(n,i)}_{1,0}\Big(m_{n,i}(0).\psi^{(n,i)}_{0,0}\Big(\frac{x_i}{m_{n,i}(0)}\Big)\Big)$
\end{center} 
for $n\geq0$, $x_i\in\D(0,\epsilon)$ and $1\leq i\leq q$. According to Proposition \ref{corlinearisation} there exists $C_1>0$ such that
\begin{eqnarray}
 & |\psi^{(n,i)}_{0,\lambda}(z)-z|\leq C_1|z|^2 \text{ for }\lambda\in\B_r\text{ and }z\in\D(0,\rho_{n,i}(\lambda)),\label{inegpsi0} & \\
 & |\psi^{(n,i)}_{1,\lambda}(w)-w|\leq C_1|w|^2 \text{ for }\lambda\in\B_r\text{ and }w\in\D(0,\rho).\label{inegpsi1} &
\end{eqnarray}
We deduce from it that the family $(u_{n,i})_{n\geq0}$ is equicontinuous and locally uniformly converges up to extraction to a holomorphic function $q_i$ on $\D(0,\epsilon)$ which verifies $q_i'(0)=1$ and $q_i(0)=0$. Set $p_i\pe q_i+z_i$.

~

\par Set now for every $x\in\B$
\begin{center}
$\begin{array}{clc}
 & \displaystyle u_{n,l,i}(x)\pe\psi^{(n,i)}_{1,S_{n,l}(x)}\Big(m_{n,i}(S_{n,l}(x)).\psi^{(n,i)}_{0,S_{n,l}(x)}\Big(\frac{x_i}{m_{n,i}(0)}\Big)\Big),& \\ & & \\
 & \displaystyle v_{n,l,i}(x)\pe m_{n,i}\Big( S_{n,l}(x)\Big)\psi_{0,S_{n,l}(x)}^{(n,i)}\Big(\frac{x_i}{m_{n,i}(0)}\Big) \textup{ and}& \\ & & \\
 & \displaystyle v_{n,i}(x_i)\pe m_{n,i}(0)\psi_{0,0}^{(n,i)}\Big(\frac{x_i}{m_{n,i}(0)}\Big).&
\end{array}$
\end{center}
To conclude it suffices to prove that $(u_{n,l,i}(x)-u_{n,i}(x_i))_{n\geq0}$ uniformly converges to $0$ on $\B$. As $S_{n,l}(\B)\subset\Omega_{n}$, the Distortion Lemma \ref{lmdistorsion}  and item 4. of Proposition \ref{propchi} give $C_2>0$, $K>1$ and $p\geq1$ such that $\left|\frac{m_{n,i}(S_{n,l}(x))}{m_{n,i}(0)}-1\right|\leq \exp(C_2\frac{n}{K^{n/p}})-1$. This together with \ref{inegpsi0} gives for $x\in\B$ and $n\geq n_0$
\begin{center}
$\displaystyle\left|v_{n,l,i}(x)-v_{n,i}(x_i)\right|\leq\left(\exp\left(\frac{nC_2}{K^{n/p}}\right)-1\right)|x_i|+ C_1\left(1+\exp\left(\frac{nC_2}{K^{n/p}}\right)\right)\frac{|x_i|^2}{|m_{n,i}(0)|}$
\end{center}
and the sequence $(v_{n,l,i}(x)-v_{n,i}(x))_{n\geq0}$ uniformly converges to $0$ on $\B$.

~

\par To conclude, we remark that
\begin{eqnarray*}
u_{n,l,i}(x)-u_{n,i}(x) & = & \ \Big(\psi_{1,S_{n,l}(x)}^{(n,i)}(v_{n,l,i}(x))-\psi_{1,0}^{(n,i)}(v_{n,l,i}(x))\Big)\\
 &  & + \ \Big(\psi_{1,0}^{(n,i)}(v_{n,l,i}(x))-\psi_{1,0}^{(n,i)}(v_{n,i}(x))\Big).
\end{eqnarray*}
 Since $S_{n,l}(\B)\subset\Omega_n$ item $1.(b)$ of Lemma \ref{lmlinbif} says that $\left|v_{n,l,i}(x)\right|\leq\rho$ and $\left|v_{n,i}(x_i)\right|\leq\rho$ for any $n\geq n_0$ and $x\in\B$. It is now clear that $(u_{n,l,i}-u_{n,i})_{n\geq1}$ uniformly converges to $0$ on $\B$. Finally, if $W\Subset\C^2\setminus\{0\}$ is a neighborhood of $\bigcup_{1\leq j\leq q}\sigma_j(\D(0,\tilde\alpha))$ (see item 2 of Lemma \ref{lmlinbif}), then $\Omega_0\times W\Subset\B_r\times\C^2\setminus\{0\}$ and $(S_{n,l}(x),\xi_{n,i}\circ S_{n,l}(x))\in \Omega_0\times W$ for any $x\in\B$ and any $n\geq n_0$ and, by continuity of $(\lambda,z)\longmapsto G_\lambda(z)$ on $\B_r\times\C^2\setminus\{0\}$, we get $M>0$ such that $\|G_\lambda(z)\|_{L^\infty(\Omega_0\times W)}\leq M$.\end{proof}

\subsection{Third step: asymptotic reduction to a dynamical data}

\par Using the $p_i$ which are given by Lemma \ref{lmconv}, we define a Radon measure on $\C^q$ by setting
\begin{center}
$\mu(B)\pe \displaystyle\int_{\D_\epsilon^q\cap B}\bigwedge_{i=1}^qdd^cG_0\circ\sigma_i\big(p_i(x_i))$ for any Borel set $B\subset\C^q$.
\end{center}

~

\begin{lm}
Let $\ell\pe\displaystyle\sum_{j=1}^{+\infty}\mu(\B_j)$. Then:
\begin{center}
$0<\mu(\D^q_\epsilon)\leq\ell\leq P(q).\mu(\D^q_\epsilon)<+\infty$.
\end{center}
\label{lmlimitsum}
\end{lm}

\begin{proof} As the functions $G_0\circ\sigma_i\big(p_i(x_i)\big)$ are continuous, the measure $\mu$ doesn't give mass to pluripolar sets. In particular, $\mu(\D_\epsilon^q)=\mu(\dot\D_\epsilon^q)$. By \ref{besicovitch}, we thus find
\begin{center}
$\displaystyle\mu(\D^q_\epsilon)=\mu(\dot\D^q_\epsilon)\leq\sum_{j=1}^{+\infty}\mu(\B_j)\leq P(q).\mu(\dot\D^q_\epsilon)=P(q).\mu(\D^q_\epsilon)$.
\end{center}
The function $G_0\circ\sigma_i\big(p_i(x_i)\big)$ only depending on the $i$-th variable, Fubini Theorem gives
\begin{center}
$\mu(\D_\epsilon^q)=\displaystyle\int_{\D_\epsilon^q}\bigwedge_{i=1}^qdd^cG_0\circ\sigma_i\big(p_i(x_i)\big)=\prod_{i=1}^q\int_{\D(0,\epsilon)}dd^cG_0\circ\sigma_i\big(p_i(z)\big)$.
\end{center}
Set now $W_i\pe p_i(\D(0,\epsilon))$ for $1\leq i\leq q$. As $p_i$ is non-constant, the set $W_i$ is an open neighborhood of $p_i(0)=z_i\in\J_0$. Thus we have $0<\mu_{f_0}(W_i)<+\infty$ and
\begin{center}
$0<\mu(\D^q_\epsilon)=\displaystyle\prod_{i=1}^q\mu_{f_0}(W_i)<+\infty$,
\end{center}
which gives the announced estimate.\end{proof}

~

\par We are now in position to end the proof of Theorem \ref{tmcourantsbif}:

~

\begin{lm}
$0\in\supp(\mu^{(q)})$. More precisely, $\displaystyle\lim_{n\rightarrow+\infty}\frac{\log\mu^{(q)}(\Omega_n)}{n}=-q\log d\neq0$.
\label{lmequivmu}
\end{lm}

\begin{proof} By Chern-Levine-Nirenberg inequalities and Lemma \ref{lmconv}, there exists a constant $M>0$ such that for any $n\geq n_0$, any $j\geq1$ and any $1\leq l\leq p$,
\begin{eqnarray*}
\int_{\B_j}\bigwedge_{i=1}^q dd^cG_{S_{n,l,j}(x)}(\tilde\xi_{n,i}\circ S_{n,l,j}(x))\leq \textup{Leb}(\B_j).M^q<+\infty.
\end{eqnarray*}
Therefore, taking the sum over $j$ and $l$ we fing for any $N\geq1$:
\begin{eqnarray*}
\sum_{l=1}^p\sum_{j\geq N+1}\int_{\B_j}\bigwedge_{i=1}^q dd^cG_{S_{n,l,j}(x)}(\tilde\xi_{n,i}\circ S_{n,l,j}(x)) & \leq &  \left(\sum_{l=1}^p\sum_{j\geq N+1}\textup{Leb}(\B_j)\right).M^q\\
 & \leq & p.M^q\sum_{j\geq N+1}\textup{Leb}(\B_j)
\end{eqnarray*}
for $n\geq n_0$. Since $\sum_{j\geq1}\mathbf{1}_{\B_j}\leq P(q).\mathbf{1}_{\dot\D^q_\epsilon}$, the series $\sum_{j\geq 1}\textup{Leb}(\B_j)$ is convergent and for any $\varepsilon>0$, there exists $N_0\geq1$ such that for any $N\geq N_0$ we find:
\begin{eqnarray*}
0\leq\sum_{l=1}^p\sum_{j\geq N+1}\int_{\B_j}\bigwedge_{i=1}^q dd^cG_{S_{n,l,j}(x)}(\tilde\xi_{n,i}\circ S_{n,l,j}(x))\leq \varepsilon.
\end{eqnarray*}
Similarly, we prove that there exists $N_1\geq N_0$ such that for $N\geq N_1$ one has:
\begin{eqnarray*}
0\leq p.\ell-\sum_{l=1}^p\sum_{j= 1}^{N}\int_{\B_j}\bigwedge_{i=1}^q dd^cG_0(\sigma_i\circ p_i(x_i))\leq \varepsilon.
\end{eqnarray*}
After Lemma \ref{lmconv}, the sequence $G_{S_{n,l,j}(x)}(\tilde\xi_{n,i}\circ S_{n,l,j}(x))$ converges to $G_0(\sigma_i\circ p_i(x_i))$ uniformly on $\B_j$. We thus have shown that
\begin{eqnarray*}
\lim_{n\rightarrow+\infty}\sum_{l=1}^p\sum_{j=1}^{+\infty}\int_{\B_j}\bigwedge_{i=1}^q dd^cG_{S_{n,l,j}(x)}(\tilde\xi_{n,i}\circ S_{n,l,j}(x))=p.\ell>0.
\end{eqnarray*}
Combined with Lemma \ref{lmTk} and Lemma \ref{lmlimitsum}, this gives a sequence $\ell_n\longrightarrow\ell$, so that
\begin{center}
$P(q)^{-1}p.d^{-q(n+k_0)}\ell_n\leq\mu^{(q)}(\Omega_n)\leq p.d^{-q(n+k_0)}\ell_n$.
\end{center}
As $\ell>0$ and as $(\Omega_n)$ is a basis of neighborhood of $0$, we have proved that $0\in\supp(\mu^{(q)})$. Moreover, the previous estimates may be rewritten:
\begin{eqnarray*}
\frac{\log(P(q)^{-1}p\ell_n)}{n}-\Big(1+\frac{k_0}{n}\Big)q\log d\leq\frac{\log\mu^{(q)}(\Omega_n)}{n}\leq\frac{\log(p\ell_n)}{n}-\Big(1+\frac{k_0}{n}\Big)q\log d.
\end{eqnarray*}
As $\ell_n$ converges, one concludes letting $n$ tend to $\infty$.\end{proof}

\section{Dimension estimates for the bifurcation measure}\par In this section, we want to  summarize the dimension estimates for the bifurcation measure which can be deduced from the previous sections.

\subsection{Hausdorff dimension of the support of $\mu_\bif$} Recall that a set $E\subset\R^k$ is said \emph{homogeneous} if for any open set $U\subset \R^k$ such that $U\cap E\neq\emptyset$, then $\dim_H(E)=\dim_H(E\cap U)$. Theorem \ref{tm2intro} combined with fine parabolic implosion techniques allows to establish the following result:

~

\begin{tm}
In the family $\rat_d$, we have the following:
\begin{enumerate}
\item The set $\supp(T_\bif^ k)\setminus\supp(T_\bif^ {k+1})$ has maximal Hausdorff dimension $2(2d+1)$,
\item the support of $T_\bif^{2d-2}$ is homogeneous and $\dim_H(\supp(T_\bif^{2d-2}))=2(2d+1)$.
\end{enumerate}
\label{tmdimH3}
\end{tm}

\begin{proof}
 Item $(1)$ is just the combination of Theorem \ref{tm1intro} and Theorem \ref{tmdimH2}. Moreover, by the Main Theorem of \cite{buffepstein}, the set $\mathfrak{M}_{2d-2}$ is dense in $\supp(T_\bif^{2d-2})$ and, using again Theorem \ref{tm1intro} and Theorem \ref{tmdimH2}, we conclude.
\end{proof}

~




\par The Lyapounov function $L:\rat_d\longmapsto\R$ induces a $p.s.h$ and continuous function $\tilde L$ on $\mathcal{M}_d$. Recall that the space $\mathcal{M}_d$ has complex dimension $2d-2$. Bassanelli and Berteloot define the \emph{bifurcation measure} in $\mathcal{M}_d$ by setting $\mu_\bif\pe(dd^c\tilde L)^{2d-2}$ (see \cite{BB1} Section 6). As the natural projection $\Pi:\rat_d\longrightarrow\mathcal{M}_d$ is a principal fibration on the complement of an analytic set (see \cite{BB1} page 226), Theorem \ref{tmdimH3} immediatly yields
\begin{center}
$\dim_H\big(\supp(\mu_\bif)\big)=2(2d-2)$
\end{center}
and the homogeneity of the set $\supp(\mu_\bif)$. This is Theorem \ref{tmprincipal}.

\subsection{Pointwise dimension of the bifurcation measure at Misiurewicz maps}\label{sectionmubifloc} Recall that, when $\mu$ is a Radon measure on $\R^k$, the \emph{upper pointwise dimension} of $\mu$ at $x\in\mathbb{R}^k$ is given by (see \cite{Mattila}):
\begin{center}
$\displaystyle\overline{\dim}_\mu(x)\pe\limsup_{r\rightarrow0}\frac{\log\mu(\B(x,r))}{\log r}$.
\end{center}
Owing to Theorem \ref{tm2intro} and Lemma \ref{lmequivmu}, we can easily give a lower bound for the upper pointwise dimension of the bifurcation measure at Misiurewicz parameters:

~

\begin{tm}
Let $[f]\in\mathcal{M}_d$ be the conjugacy class of a $(2d-2)$-Misiurewicz rational map $f\in\rat_d$. Assume that $[f]$ is not a singular point of $\mathcal{M}_d$. Then the upper pointwise dimension of $\mu_\bif$ at $[f]$ satisfies
\begin{center}
$\displaystyle\overline{\dim}_{\mu_\bif}[f]\geq (2d-2)\cdot\frac{\log d}{\liminf\limits_{n\rightarrow+\infty}\frac{\log\mathfrak{m}_n^+}{n}}>0$,
\end{center}
where $k_0\geq 1$ is the least integer such that $P^{k_0}(f)$ is an $f$-hyperbolic set and
\begin{center}
$\displaystyle\mathfrak{m}_n^+\pe\max_{1\leq i\leq 2d-2}\left|(f^n)'(f^{k_0}(c_i))\right|$.
\end{center}
\label{cordimmuq}
\end{tm}

\begin{proof} Let $f$ be a $(2d-2)$-Misiurewicz map and let $V_f$ be the local $(2d-2)$-dimensional submanifold of $\rat_d$ containing $f$ which is transversal to the orbit $\mathcal{O}(f)\subset\rat_d$ of $f$ under the action of $\textup{Aut}(\p^1)$ (see for example \cite{BB1} page 226). Let $\B(f,r_0)\subset V_f$ be a ball centered at $f$. Recall that, since $P^{k_0}(f)$ is an $f$-hyperbolic set, there exists constants $0<A<B<1$ such that $A^n\leq|(f^n)'(f^{k_0}(c_i))|^{-1}\leq B^n$ for any $n\geq0$ and any $1\leq i\leq 2d-2$ and that we have defined a basis of neighborhood of $f$ in $V_f$ by defining $\Omega_n$ as the component of
\begin{center}
$\chi^{-1}\D(0,\epsilon/(f^n)'(f^{k_0}(c_1)))\times\cdots\times\D(0,\epsilon/(f^n)'(f^{k_0}(c_{2d-2})))$
\end{center}
containing $f$, where $\chi$ is the activity map defined in section \ref{sectionactivitymap}. After items $3$ and $4$ of Proposition \ref{propchi}, there exists a constant $C>0$ such that $\B(f,C/\mathfrak{m}_n^+)\subset\Omega_n$. If we set $\mu\pe (dd^cL|_{V_f})^{2d-2}$, it thus comes
\begin{eqnarray*}
\frac{\log\mu\big(\B(f,C/\mathfrak{m}_{n}^+)\big)}{\log (C/\mathfrak{m}_{n}^+)}\geq \frac{\log\mu(\Omega_{n})}{\log (C/\mathfrak{m}_{n}^+)}=\frac{\log\mu(\Omega_{n})}{-n}\cdot\frac{-n}{\log C -\log\mathfrak{m}_{n}^+}.
\end{eqnarray*}
By Theorem \ref{tmddcg} and Lemma \ref{lmequivmu}, taking the $\limsup$ over $n$, we find:
\begin{eqnarray*}
\limsup_{r\rightarrow0}\frac{\log\mu(\B(f,r))}{\log r}\geq\limsup_{n\rightarrow+\infty}\frac{\log\mu(\B(f,C/\mathfrak{m}_{n}^+))}{\log C/\mathfrak{m}_{n}^+}\geq (2d-2)\cdot\log d\cdot\frac{1}{\liminf\limits_{n\rightarrow+\infty}\frac{\log\mathfrak{m}_n^+}{n}}.
\end{eqnarray*}
The natural projection $\Pi:V_f\longrightarrow\mathcal{M}_d$ being a finite branched cover, we get
\begin{center}
$\overline{\dim}_{\mu_\bif}[f]=\displaystyle\limsup_{r\rightarrow0}\frac{\log\mu(\B(f,r))}{\log r}$.
\end{center}
Since $|(f^n)'(f^{k_0}(c_i))|\leq A^{-n}$ with $A<1$, we can conclude that $\overline{\dim}_{\mu_\bif}[f]>0$.
\end{proof}

~

\par Recall that the family $\rat_d$ of all degree $d$ rational maps is a quasiprojective variety of $\p^{2d+1}$ which is connected (see \cite{BB1} Section 1.1). We denote by $\omega$ the Fubini-Study form of $\p^{2d+1}$ restricted to $\rat_d$.

~

\begin{rem}
The proof should be similar when the point $[f]$ is singular. Let:
\begin{center}
$\sigma_{T_\bif^k}\pe\Pi_*(T_\bif^k\wedge\omega^{2d-2-k})=(\Pi_*T_\bif)^k\wedge(\Pi_*\omega)^{2d-2-k}$
\end{center}
where $\Pi:\rat_d\longrightarrow\mathcal{M}_d$ is the quotient map. As in the classical situation, we call this measure the \emph{trace measure} of $(\Pi_*T_\bif)^k$. Theorem \ref{cordimmuq} may also generalize for $k$-Misiurewicz maps in the following way:
\begin{eqnarray*}
\overline{\dim}_{\sigma_{T_\bif^k}}[f]\geq k\cdot\frac{\log d}{\liminf\limits_{n\rightarrow\infty}\frac{\log\mathfrak{m}_n^+}{n}}>0.
\end{eqnarray*}
\end{rem}

~ 

In the case of strictly postcritically finite rational maps, Theorem \ref{cordimmuq} can be precised:

~

\begin{cor}
Let $[f]\in\mathcal{M}_d$ be the conjugacy class of a strictly postcritically finite rational map $f\in\rat_d$ which is not a flexible Latt\`es map and such that $[f]$ is not a singular point of $\mathcal{M}_d$. Let $\alpha_j$ be the multipliers of the $p_j$-repelling cycles capturing the critical point $c_j$ and let $p\pe \gcd(p_j)$. Then:
\begin{center}
$\displaystyle\overline{\dim}_{\mu_\bif}[f]\geq (2d-2)\cdot\frac{\log d}{\log\max\limits_{1\leq i\leq2d-2}|\alpha_i|^{p/p_i}}$.
\end{center}
In particular, if $f$ is a rigid (i.e. non-flexible) Latt\`es map, then $\overline{\dim}_{\mu_\bif}(f)=2 (2d-2)$.
\end{cor}

\begin{proof}
Let $f$ be a strictly postcritically finite degree $d$ rational map which is not a flexible Latt\`es map, then the same proof as for Theorem \ref{cordimmuq} gives the wanted estimate, since
 \begin{center}
$(f^{n_jp})'(f^k_0(c_i))=\big((f^p)'(f^{k_0}(c_i))\big)^{n_j}=(\alpha_i^{p/p_i})^{n_j}$.
\end{center}
\par If $f$ is a Latt\`es map, the multiplier of any $n$-cycles has modulus $(\sqrt{d})^n$ (see \cite{Milnor} Corollary 3.9) and the result follows.\end{proof}

~

\par The proof of Theorem \ref{tm3intro} is now obvious: there exists a projective variety $V\subset\rat_d$ such that $\mathcal{M}_d\setminus\Pi(V)$ is a complex manifold. Taking $\mathfrak{M}\pe\Pi(\mathfrak{M}_{2d-2}\setminus V)$ gives directly Theorem \ref{tm3intro}.

\section{In the family of degree $d$ polynomials}

\par Let $\poly_d$ be the family of all degree $d$ complex polynomials. The group $\Aut(\C)$ acts by conjugacy on $\poly_d$ and we denote by $\mathcal{P}_d\pe \poly_d/\Aut(\C)$ the quotient space. For $c=(c_1,\ldots,c_{d-2})\in\C^{d-2}$ and $a\in\C$ we denote $\lambda\pe(c,a)\in\C^{d-1}$ and we set
\begin{center}
$p_\lambda(z)\pe \displaystyle\frac{1}{d}z^d+\sum_{j=2}^{d-1}(-1)^{d-j}\frac{\sigma_{d-j}(c)}{j}z^j+a^d$.
\end{center}
This parametrization was introduced by Branner and Hubbard in \cite{BH} to study the connectedness locus $\mathcal{C}_d$ of the family $\mathcal{P}_d$. This family has $d-1$ marked critical points $c_0(\lambda)\pe 0,c_1(\lambda)\pe c_1,\ldots, c_{d-2}(\lambda)\pe c_{d-2}$ and admits a natural compactification as $\p^{d-1}$. The following Proposition (see \cite{BH} section 2 and Corollary 3.7 or \cite{favredujardin} Proposition 5.1, Proposition 6.2 and Proposition 6.14 or \cite{BB2} section 4.2) summarizes the interest of this parametrization:

~

\begin{prop}
\begin{enumerate}
\item The natural projection $\Pi:\C^{d-1}\longrightarrow\mathcal{P}_d$ is a degree $d(d-1)$ analytic branched cover,
\item The loci $\mathcal{B}_i\pe\{\lambda$ / $(p^n_\lambda(c_i(\lambda)))_{n\geq1}$ is bounded in $\C\}$ accumulate at infinity of $\C^{d-1}$ in $\p^{d-1}$ on codimension $1$ algebraic sets $\Gamma_i$ of the hyperplan $\p_{\infty}=\p^{d-1}\setminus\C^{d-1}$ which intersect two-by-two transversally. As a consequence, $\mathcal{C}_d$ is compact in $\C^{d-1}$,
\item The support of the bifurcation measure $\mu_\bif\pe T_\bif^{d-1}$ is the Shilov boundary $\mathcal{C}_d$. 
\end{enumerate}
\label{propPolyd}
\end{prop}

\par In the family $(p_\lambda)_{\lambda\in\C^{d-1}}$, the main results stated throughout the paper hold. Let us mention that for proving Lemma \ref{parabolicimplosion}, we don't use the fact that we have a family of rational maps. Therefore, Lemma \ref{parabolicimplosion} and Theorem \ref{tmdimH2} remain valid in this context. This again implies that $\supp(T_\bif^k)\setminus\supp(T_\bif^{k+1})$ has maximal Hausdorff dimension $2(d-1)$ and, for $k=d-1$, that the support of the measure $\mu_\bif$ is homogeneous and has maximal Hausdorff dimension. Owing to Proposition \ref{propPolyd}, this is exactly Theorem \ref{tmshilov}.

~

\par The following Lemma allows us to generalize Theorem \ref{tm2intro} for polynomials having critical points preperiodic to repelling cycles: to have transversality at $p_{\lambda_0}$, you don't need any asumptions on the critical points of $p_{\lambda_0}$ which are not preperiodic to repelling cycles. In particular, the formalism of good families is not anymore relevent for geometrically finite Misiurewicz polynomials:

~

\begin{lm}
Let $\lambda_0\in\C^{d-1}$  and $1\leq k\leq d-1$. Assume that $p_{\lambda_0}$ has $k$ critical points $c_{j_1}(\lambda_0),\ldots,c_{j_k}(\lambda_0)$ preperiodic to repelling cycles with $p_{\lambda_0}^{k_0+n_i}(c_{j_i}(\lambda_0))=p_{\lambda_0}^{k_0}(c_{j_i}(\lambda_0))$. Set $\chi_i(\lambda)\pe p_{\lambda}^{k_0+n_i}(c_{j_i}(\lambda))-p_{\lambda}^{k_0}(c_{j_i}(\lambda))$ and let $\chi$ be the activity map of $p_{\lambda_0}$:
\begin{eqnarray*}
\chi:\C^{d-1} & \longrightarrow & \C^k\\
\lambda & \longmapsto & \big(\chi_1(\lambda),\ldots,\chi_k(\lambda)\big).
\end{eqnarray*}
Then $\textup{codim}\ \chi^{-1}\{0\}=k$.
\label{lmpolynomes}
\end{lm}

\begin{proof} First remark that the activity map exhibited in the Lemma coincides with the one introduced in section \ref{sectionactivitymap} in a neighborhood of $\lambda_0$. Consider the algebraic set
\begin{center}
$H_i\pe\{\lambda\in\C^{d-1}$ / $p_{\lambda}^{k_0+n_i}(c_{j_i}(\lambda))=p_{\lambda}^{k_0}(c_{j_i}(\lambda))\}=\chi^{-1}_i\{0\}$.
\end{center}
Since $H_i\subset\mathcal{B}_{j_i}$, it is an hypersurface which extends as an hypersurface of $\p^{d-1}$, still denoted by $H_i$. Proposition \ref{propPolyd} implies that $H_1\cap\cdots\cap H_k\cap \p_{\infty}\subset\Gamma_{j_1}\cap\cdots\cap \Gamma_{j_k}\cap \p_{\infty}$. Using again Proposition \ref{propPolyd}, we get $\textup{codim}\ H_1\cap\cdots\cap H_k\cap \p_\infty\geq k+1$. As $\p_\infty$ is a hypersurface, we have $\textup{codim}\ H_1\cap\cdots\cap H_k\geq k$. Since $\lambda_0\in H_1\cap \cdots\cap H_k$, the set $H_1\cap \cdots\cap H_k$ is not empty and thus $\textup{codim}\ H_1\cap\cdots\cap H_k\leq k$.\end{proof}

~

\par Due to the above Lemma, applying Theorem \ref{tmcourantsbif} we find:

~

\begin{tm}
Let $\lambda_0\in\C^{d-1}$  and $1\leq k\leq d-1$. Assume that $p_{\lambda_0}$ has $k$ critical points which are preperiodic to repelling cycles, then $\lambda_0\in\supp(T_\bif^k)$.
\end{tm}

\bibliographystyle{short}
\bibliography{biblio}

\end{document}